\documentclass[11pt]{amsart}
\date{September 21, 2010}
\usepackage{graphicx, color, enumerate}
\usepackage{amsmath, amssymb, amscd, mathrsfs}
\usepackage[varg]{txfonts}
\title[Isometric Immersions of $\H^2$ into $\H^3$]{%
  Isometric Immersions of the Hyperbolic Plane\\
  into the Hyperbolic Space
}
\author[A. Honda]{Atsufumi Honda}
\address{%
   Department of Mathematics,
   Graduate School of Science and Engineering\endgraf
   Tokyo Institute of Technology\endgraf
   O-okayama, Meguro, Tokyo 152-8551\endgraf
   Japan
}
\email{10d00059@math.titech.ac.jp}
\usepackage{amsthm}
\theoremstyle{plain}
 \newtheorem{theorem}{Theorem}[section]
 \newtheorem{introtheorem}{Theorem}
   
 \newtheorem{proposition}[theorem]{Proposition}
 
 \newtheorem{lemma}[theorem]{Lemma}
 \newtheorem{corollary}[theorem]{Corollary}
 \theoremstyle{remark}
 \newtheorem{definition}[theorem]{Definition}
 \newtheorem{remark}[theorem]{Remark}
 \newtheorem*{acknowledgements}{Acknowledgements}
 \newtheorem{example}[theorem]{Example}
\numberwithin{equation}{section}

\newcommand{\N}{\boldsymbol{N}}
\newcommand{\Z}{\boldsymbol{Z}}

\newcommand{\R}{\boldsymbol{R}}
\newcommand{\C}{\boldsymbol{C}}
\newcommand{\Herm}{\operatorname{Herm}}
\newcommand{\GL}{\operatorname{GL}}
\newcommand{\SL}{\operatorname{SL}}
\newcommand{\PSL}{\operatorname{PSL}}
\newcommand{\U}{\operatorname{U}}
\newcommand{\SU}{\operatorname{SU}}
\newcommand{\SO}{\operatorname{SO}}
\newcommand{\bmath}[1]{\mbox{\boldmath $#1$}}
\renewcommand{\H}{{\bmath{H}}}
\renewcommand{\S}{{\bmath{S}}}
\renewcommand{\P}{\bmath{P}}
\renewcommand{\L}{{\bmath{L}}}
\newcommand{\B}{{\bmath{B}}}
\newcommand{\D}{{\bmath{D}}}
\newcommand{\m}{\mu}
\newcommand{\trace}{\operatorname{trace}}
\newcommand{\inner}[2]{\left\langle{#1},{#2}\right\rangle}
\renewcommand{\Re}{\operatorname{Re}}
\renewcommand{\Im}{\operatorname{Im}}
\newcommand{\vect}[1]{\boldsymbol{#1}}
\newcommand{\Vect}[1]{\overrightarrow{#1}}
\renewcommand{\l}{\lambda}
\newcommand{\LH}{L\H^3}
\newcommand{\LR}{L\R^3}

\newcommand{\LoH}{L_{o}\H^3}
\newcommand{\bihol}{\stackrel{\sim}{\longrightarrow}}
\begin{document}
\maketitle
\begin{abstract}
 In this paper, we parametrize the space of isometric immersions 
 of the hyperbolic plane into the hyperbolic $3$-space 
 in terms of null-causal curves in the space of oriented geodesics. 
 Moreover, we characterize ``ideal cones'' 
 (i.e., cones whose vertices are on the ideal boundary) 
 by behavior of their mean curvature.
\end{abstract}
\section*{Introduction}
\begingroup
\renewcommand{\theequation}{\arabic{equation}}
Consider isometric immersions of 
$\tilde{\Sigma}^n(c)$ into $\tilde{\Sigma}^{n+1}(c)$,
where $\tilde{\Sigma}^m(c)$ denotes the simply connected 
$m$-dimensional space form of constant sectional curvature $c$.
Such immersions are only cylinders \cite{HN} 
in the Euclidean case $(c=0)$.
In the spherical case $(c>0)$,
such immersions are only totally geodesic embeddings \cite{OS}.
On the other hand, in the hyperbolic case $(c<0)$, 
it is well-known that there are nontrivial examples of 
such isometric immersions \cite{Nomizu,Ferus,AbeHaas}
(see Figure~\ref{fig:Nomizu} for the case of $n=2$).

\begin{figure}[htb]
 \begin{tabular}{{c@{\hspace{10mm}}c@{\hspace{10mm}}c@{\hspace{10mm}}c}}
  \resizebox{2.5cm}{!}{\includegraphics{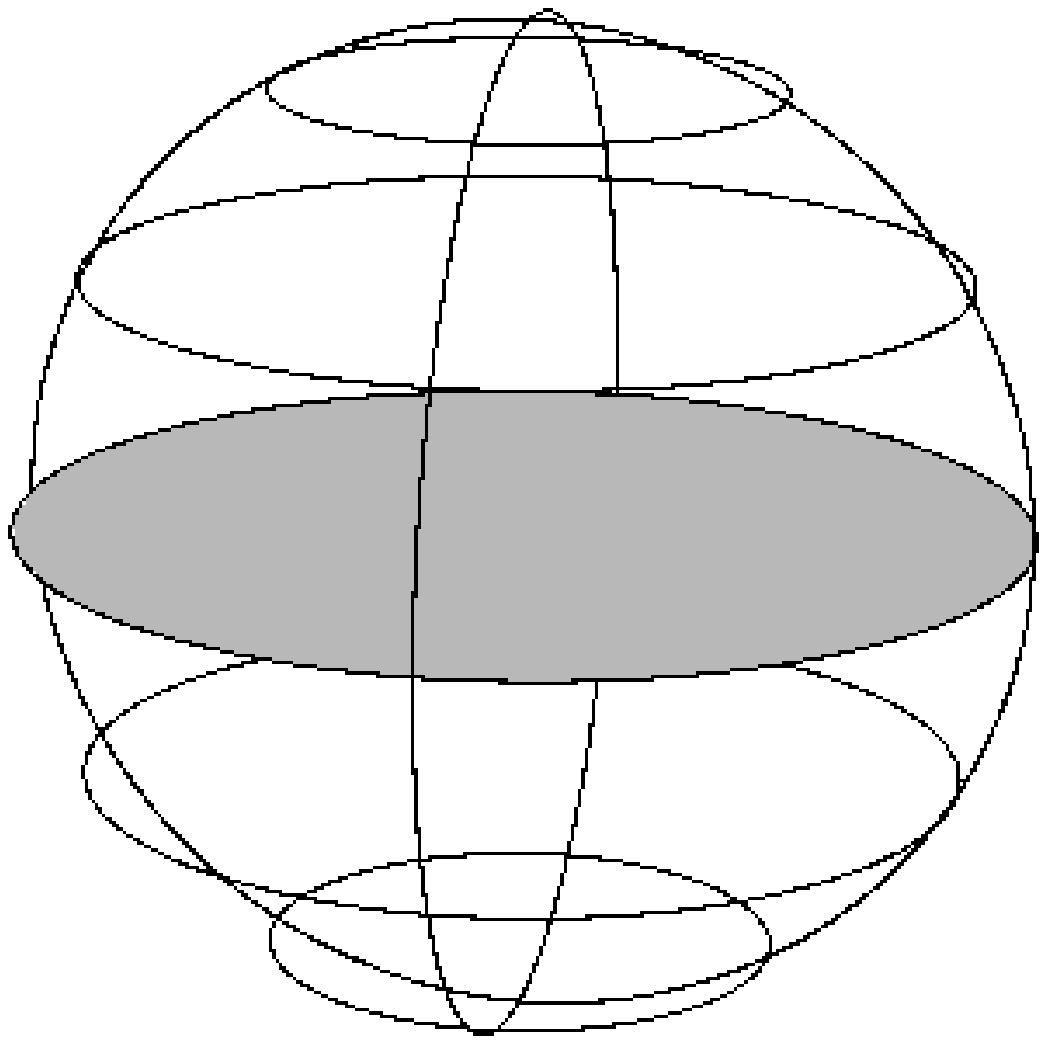}} &
  \resizebox{2.5cm}{!}{\includegraphics{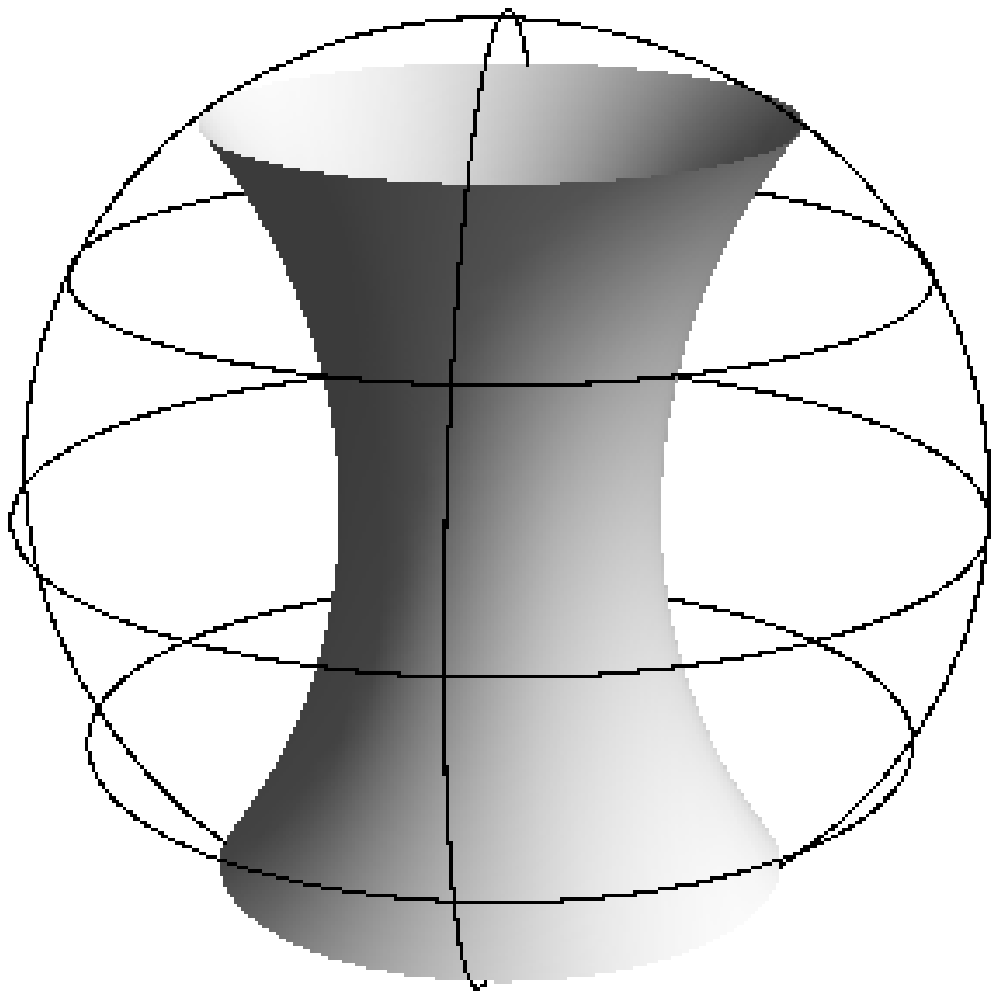}} &
  \resizebox{2.5cm}{!}{\includegraphics{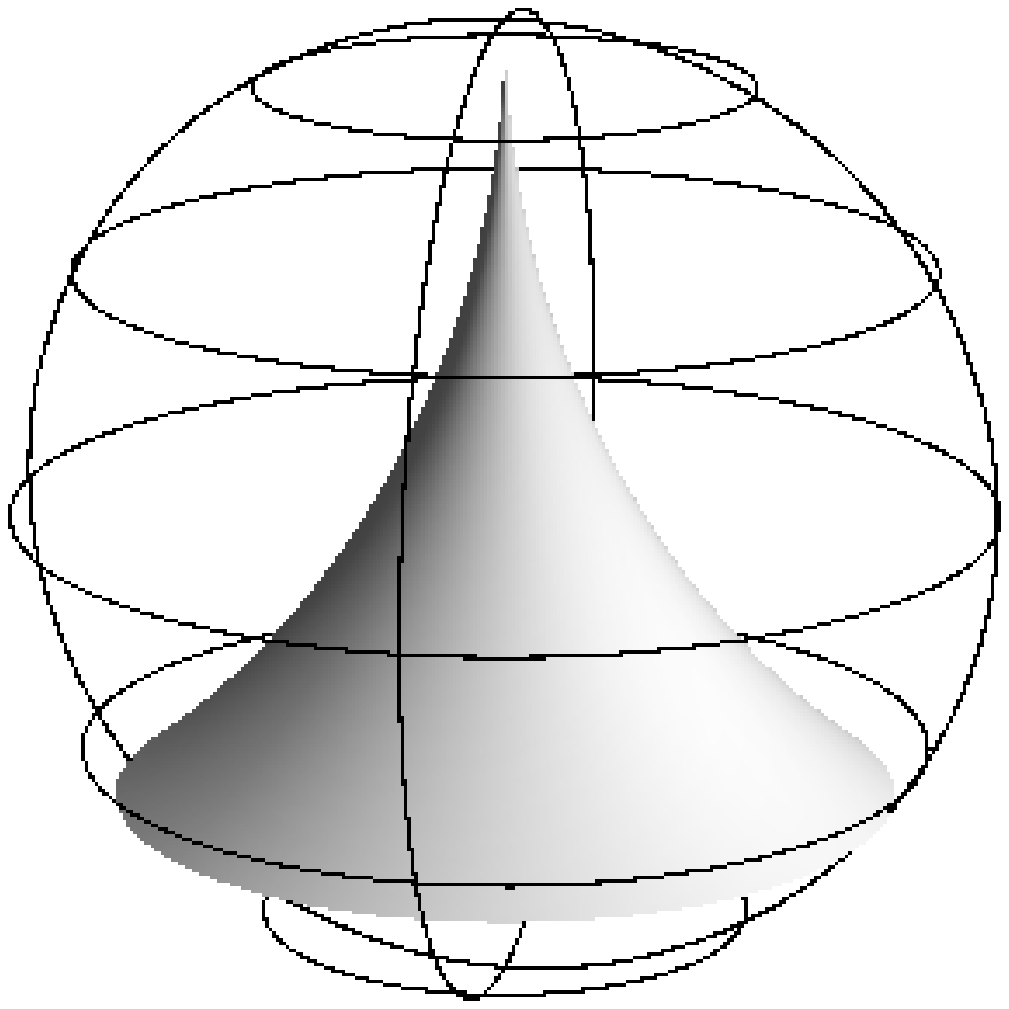}}&
  \resizebox{2.5cm}{!}{\includegraphics{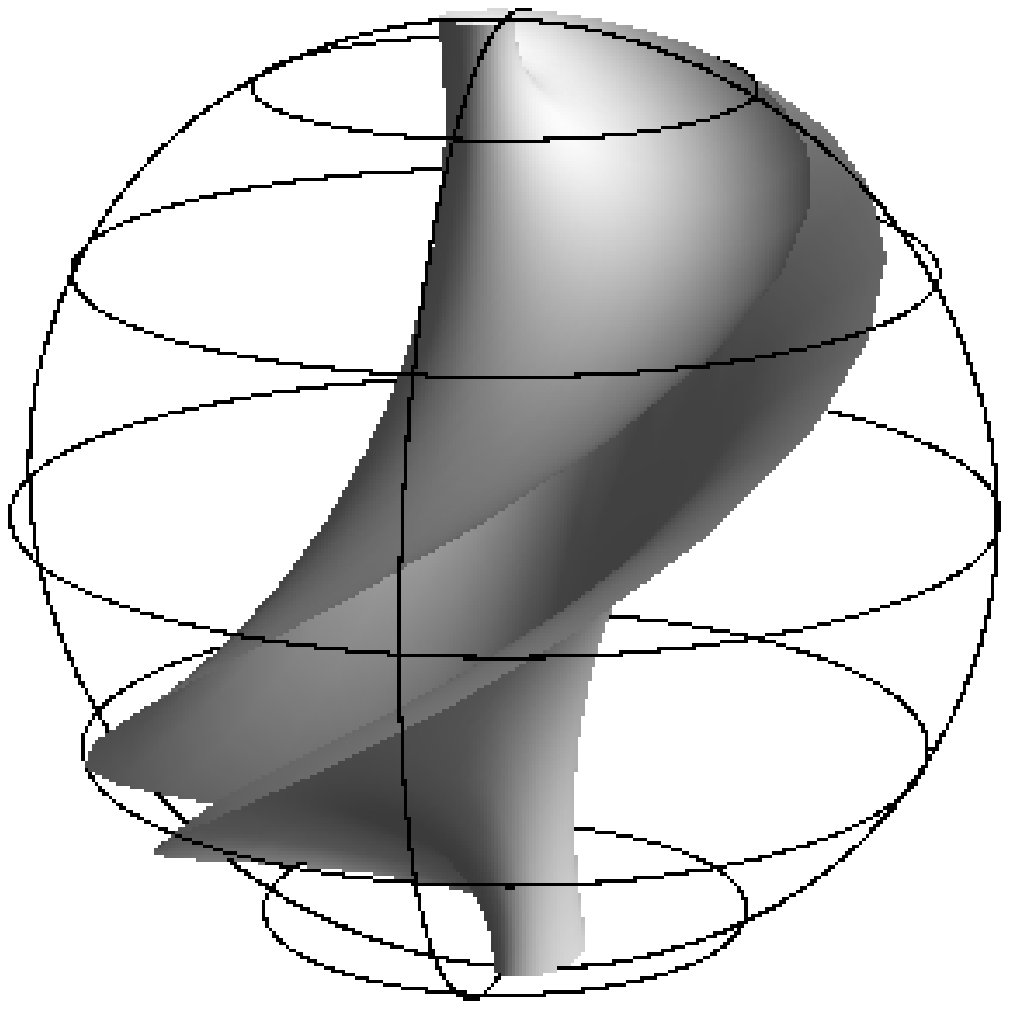}} \\
  {\footnotesize  (A) totally geodesic} &
  {\footnotesize  (B) Example \ref{ex:1}} &
  {\footnotesize  (C) Example \ref{ex:2}} &
  {\footnotesize  (D) Example \ref{ex:3}}
 \end{tabular}
 \caption{Examples constructed by Nomizu \cite{Nomizu} (see Section 3).}
 \label{fig:Nomizu}
\end{figure}

\noindent
We denote by $\H^n=\tilde{\Sigma}^n(-1)$ the $n$-dimensional hyperbolic space, 
that is, the complete simply connected and connected Riemannian manifold 
of constant curvature $-1$.
Nomizu \cite{Nomizu} and Ferus \cite{Ferus} showed that,
for a given $C^{\infty}$ totally geodesic foliation of codimension $1$ in $\H^n$,
there is a family of isometric immersions of $\H^n$ into $\H^{n+1}$ 
without umbilic points such that, for each immersion, the foliation defined by 
its asymptotic distribution coincides with the given foliation.
Furthermore, Abe, Mori and Takahashi \cite{AbeMoriTakahashi} 
parametrized the space of isometric immersions of $\H^n$ into $\H^{n+1}$
by a family of properly chosen countably many $\R^n$-valued functions.

In this paper, we shall give another parametrization in the case of $n=2$:
we represent isometric immersions of $\H^2$ into $\H^3$ 
by curves in the space $\LH$ of oriented geodesics in $\H^3$.
Moreover, we characterize certain asymptotic behavior of such immersions 
in terms of their mean curvature.

More precisely, an isometric immersion of $\H^2$ into $\H^3$ 
is a complete {\it extrinsically flat\/} surface in $\H^3$,
that is, a complete surface whose extrinsic curvature vanishes.
It is known that a complete extrinsically flat surface is {\it ruled\/},
i.e., a locus of a $1$-parameter family of geodesics in $\H^3$ 
\cite{Portnoy} (see Proposition \ref{prop:Portnoy}).
Hence, we shall deal with extrinsically flat ruled surfaces:\ 
{\it developable\/} surfaces in $\H^3$.
On the other hand, it is well-known that the space of oriented geodesics $\LH$
has two significant geometric structures: 
the natural complex structure $J$ \cite{Hitchin,GG} 
and the para-complex structure $P$ \cite{Kaneyuki,Kanai,Kimura}.
Recently, Salvai \cite{Salvai} determined the family of metrics 
$\{\mathcal{G}_{\theta}\}_{\theta \in S^1}$ 
each of which is invariant under the action of the identity component 
of the isometry group of $\H^3$.
Each metric $\mathcal{G}_{\theta}$ is of neutral signature, 
K\"ahler with respect to $J$ and para-K\"ahler with respect to $P$.
In this paper,
we especially focus on two neutral metrics 
$\mathcal{G}^{\mathfrak{r}}=\mathcal{G}_0$ 
and $\mathcal{G}^{\mathfrak{i}}=\mathcal{G}_{\pi/2}$ 
in $\{\mathcal{G}_{\theta}\}_{\theta \in S^1}$.
In Section~\ref{sec:geomstr}, 
we shall investigate the relationships among 
$J$, $P$,  $\{\mathcal{G}_{\theta}\}_{\theta \in S^1}$ 
and the canonical symplectic form on $\LH$,
and give a characterization of $\mathcal{G}^{\mathfrak{i}}$ 
and $\mathcal{G}^{\mathfrak{r}}$ (Proposition \ref{prop:symplectic}).
In Section~\ref{sec:null}, 
we introduce a representation formula for developable surfaces in $\H^3$ 
in terms of {\it null-causal curves\/} (Proposition \ref{prop:representation}):
\begin{introtheorem}\label{thm:null}
  A curve in $\LH$ which is null with respect to $\mathcal{G}^{\mathfrak{i}}$ 
  and causal with respect to $\mathcal{G}^{\mathfrak{r}}$ 
  generates a developable surface in $\H^3$.
  Conversely, any developable surface generated by complete 
  geodesics in $\H^3$ is given in this manner.
\end{introtheorem}
\noindent
Here, a regular curve in a pseudo-Riemannian manifold 
is called {\it null\/} (resp.\ {\it causal\/}) if every tangent vector 
gives null (resp.\ timelike or null) direction.
In Section~\ref{sec:exponential}, 
we shall investigate curves in $\LH$ which are null with respect to 
both $\mathcal{G}^{\mathfrak{r}}$ and $\mathcal{G}^{\mathfrak{i}}$.
Such curves generate cones whose vertices are on the ideal boundary, 
which we call {\it ideal cones\/} (Proposition \ref{prop:vertex}).
On the other hand,
on each asymptotic curve $\gamma$ on a complete developable surface,
the mean curvature is proportional to $e^{\pm t}$ or $1/\cosh t$,
where $t$ denotes the arc length parameter of $\gamma$ (Lemma \ref{lem:Massey}).
Based on this fact, a complete developable surface is 
said to be {\it of exponential type\/},
if the mean curvature is proportional to $e^{\pm t}$ on each asymptotic curve
in the non umbilic point set (see Definition \ref{def:exptype}).
Then we have the following
\begin{introtheorem}\label{thm:exponential}
  A real-analytic developable surface of exponential type is an ideal cone.
\end{introtheorem}
\noindent
The assumption of ``real-analyticity'' cannot be removed (see Example \ref{ex:NRA}).

As mentioned before,
complete flat surfaces in the Euclidean 3-space $\R^3$ are only cylinders.
However, if we admit {\it singularities}, 
there are a lot of interesting examples.
Murata and Umehara \cite{MurataUmehara} investigated 
the global geometric properties of 
a class of flat surfaces with singularities in $\R^3$, so-called {\it flat fronts\/}.
On the other hand, there is another generalization of 
ruled (resp.\ developable) surfaces in $\R^3$:
{\it horocyclic\/} (resp.\ {\it horospherical flat horocyclic\/}) surfaces in $\H^3$ 
(for more details, see \cite{IzumiyaSajiTakahashi,TakizawaTsukada}).

\begin{acknowledgements}
     Thanks are due to Kotaro Yamada, author's advisor, 
     for many helpful comments and discussions. 
     The author also would like to thank Masaaki Umehara 
     for his intensive lecture on surfaces with singularities
     at Kumamoto University on June, 2009,
     which made the author interested in current subjects.
     He also thanks a lot to Jun-ichi Inoguchi for valuable discussions 
     and constant encouragement.
     Finally, the author expresses gratitude to 
     Masahiko Kanai, Soji Kaneyuki and Yu Kawakami for their helpful comments.
\end{acknowledgements}
\endgroup

\section{Preliminaries}
\label{sec:prelim}

\subsection{Hyperbolic $3$-space}
\label{sec:H3}
\hspace{2mm}

We denote by $\L^4$ the Lorentz-Minkowski $4$-space with the Lorentz metric 
\[
  \inner{ {}^t (x_{0}, x_{1}, x_{2}, x_{3}) }{ {}^t (y_{0}, y_{1}, y_{2}, y_{3}) } 
  = -x_{0}y_{0}+x_{1}y_{1}+x_{2}y_{2}+x_{3}y_{3},
\]
where ${}^t$ denotes the transposition. 
Then the hyperbolic $3$-space is given by
\begin{equation}\label{eq:LM-model}
  \H^{3}=\left\{\left. \vect{x}={}^t (x_{0}, x_{1}, x_{2}, x_{3})\in \L^{4} \,\right|\, 
  \inner{\vect{x}}{\vect{x}}=-1, ~x_{0}>0 \right\}
\end{equation}
with the induced metric from $\L^{4}$, 
which is a complete simply connected and connected Riemannian $3$-manifold 
with constant sectional curvature $-1$. 
We identify $\L^4$ with the set of $2\times 2$ Hermitian matrices 
${\rm Herm}(2)=\{X^{\ast}=X\}~(X^{\ast}:= {}^t\bar{X})$ by
\[
  \L^4 \ni {}^t(x_0,x_1,x_2,x_3) \longleftrightarrow \left(
  \begin{array}{cc}
    x_0+x_3 & x_1+i x_2 \\
    x_1-i x_2 & x_0-x_3 
  \end{array}
  \right) \in \Herm(2)
\]
with the metric 
\[
  \inner{X}{Y}=-\frac{1}{2}\trace(X\tilde{Y}),\qquad \inner{X}{X}=-\det X, 
\]
where $\tilde{Y}$ is the cofactor matrix of $Y$.
Under this identification, the hyperbolic $3$-space $\H^3$ is represented as
\begin{equation}\label{eq:Herm-model}
  \H^3=\left\{\left. p \in \Herm(2) \,\right|\, \det p=1,\, \trace p>0 \right\}.
\end{equation}
We call this realization of $\H^3$ the {\it Hermitian model}.
We fix the basis $\{ \sigma_0, \sigma_1, \sigma_2, \sigma_3 \}$ of $\Herm(2)$ as 
\begin{equation}
\label{eq:basis}
  \sigma_0 = {\rm id}, \quad
  \sigma_1= \left(\begin{array}{cc}
  0 & 1 \\
  1 & 0 
  \end{array}
  \right),
  \quad
  \sigma_2= \left(\begin{array}{cc}
  0 & -i \\
  i & 0 
  \end{array}
  \right),
  \quad
  \sigma_3= \left(\begin{array}{cc}
  1 & 0 \\
  0 & -1 
  \end{array}
  \right). 
\end{equation}
In the Hermitian model, the cross product at $T_p\H^3$ is given by
\begin{equation}\label{eq:cross}
  X \times Y= \frac{i}{2}(Xp^{-1}Y-Yp^{-1}X), 
\end{equation}
for $X,Y\in T_p\H^3$ (cf. \cite[(3 - 1)]{KRSUY}).
The special linear group $\SL(2,\C)$ acts isometrically and transitively on $\H^3$ by
\begin{equation}\label{eq:isom}
  \H^3 \ni p \longmapsto apa^{\ast} \in \H^3,  
\end{equation}
where $a \in \SL(2,\C)$.
The isotropy subgroup of $\SL(2,\C)$ at $\sigma_0$ is 
the special unitary group $\SU(2)$. Therefore we can identify 
\begin{eqnarray*}
  \H^3 
  = \SL(2,\C)/\SU(2)
  = \left\{\left. aa^{\ast} \,\right|\, a \in \SL(2,\C) \right\}
\end{eqnarray*}
in the usual way. 
Moreover, the identity component of the isometry group ${\rm Isom}_0(\H^3)$ 
is isomorphic to $\PSL(2,\C):= \SL(2,\C)/\{\pm \text{id} \}$.

\subsection{The unit tangent bundle}
\label{sec:UH3}
\hspace{2mm}

We denote by $U\H^3$ the unit tangent bundle of $\H^3$, 
which can be identified with
\[
  U\H^3=\left\{ (p,v) \in \Herm(2) \times \Herm(2) 
  \,\left|\, \begin{array}{cc} \det p=-\det v=1,\\ \trace p>0,
  ~ \inner{p}{v}=0 \end{array} \right. \right\}.
\]
The projection 
\begin{equation}\label{eq:pi}
  \pi : U\H^3 \ni (p,v) \longmapsto p \in \H^3
\end{equation}
gives a sphere bundle.
The tangent space at $(p,v) \in U\H^3$ can be written by
\begin{equation}\label{eq:tangUH3}
  T_{(p,v)}U\H^3=\left\{ (X,V) \in \Herm(2)\times\Herm(2) \,\left|\, 
  \begin{array}{cc} \inner{p}{X}=\inner{v}{V}=0,\\
  \inner{p}{V}=-\inner{X}{v} \end{array} \right.\right\}.
\end{equation}
The {\it canonical contact form} $\Theta$ on $U\H^3$ is given by
\begin{equation}\label{eq:cano-cont}
  \Theta_{(p,v)}(X,V)=\inner{X}{v}=-\inner{p}{V},\qquad (X,V)\in T_{(p,v)}U\H^3.
\end{equation}
The isometric action of $\SL(2,\C)$ on $\H^3$ as in \eqref{eq:isom} induces 
a transitive action on $U\H^3$ as
\[
  U\H^3 \ni (p,v) \longmapsto (apa^{\ast},ava^{\ast}) \in U\H^3,  
\]
where $a \in \SL(2,\C)$.
The isotropy subgroup of $\SL(2,\C)$ at $(\sigma_0,\sigma_3) \in U\H^3$ is
\[
  \left\{\left. \left( \begin{array}{cc} e^{i \theta} & 0\\ 0 & e^{-i \theta} \end{array} \right) 
  \,\right|\, \theta \in \R/2\pi\Z \right\}
\]
which is isomorphic to the unitary group $\U(1)$,
where $\sigma_0$ and $\sigma_3$ are as in \eqref{eq:basis}.
Hence we have
\begin{equation}\label{eq:unitan}
  U\H^3
  =\SL(2,\C)/\U(1)
  = \left\{\left. (aa^{\ast},a\sigma_3 a^{\ast}) \,\right|\, a \in {\rm SL}(2,\C) \right\}.
\end{equation}

\subsection{The space of oriented geodesics}
\label{sec:LH3}
\hspace{2mm}

The space $\LH$ of oriented geodesics in $\H^3$ is
defined as the set of equivalence classes of unit speed geodesics in $\H^3$.
Here, two unit speed geodesics $\gamma_1(t),\,\gamma_2(t)$ 
in $\H^3$ are said to be {\it equivalent\/}
if there exists $t_0 \in \R$ such that $\gamma_1(t+t_0)=\gamma_2(t)$.
We denote by $[\gamma]$ the equivalence class represented by $\gamma(t)$. 
The set $\LH$ has a structure of a smooth $4$-manifold.
Moreover, if we denote by $\SO^+(1,1)$ the restricted Lorentz group, the projection
\begin{equation}\label{eq:pihat}
  \hat{\pi}: U\H^3 \ni (p,v) \longmapsto [\gamma_{p,v}] \in \LH
\end{equation}
defines an $\SO^+(1,1)$-bundle,
where $\gamma_{p,v}$ is the geodesic starting at $p \in \H^3$ 
with the initial velocity $v\in T_p\H^3$. 

\subsubsection{The natural complex structure and a holomorphic coordinate system}
\hspace{2mm}

Hitchin \cite{Hitchin} constructed the natural complex structure $J$ on 
$\LH$ ({\it minitwistor construction}).
Here, we introduce a local holomorphic coordinate system 
$(\m_1,\m_2)$ of the complex surface $(\LH,J)$ \cite{GG}.
We denote by $\partial \H^3$ the ideal boundary of $\H^3$,
that is, the set of asymptotic classes of oriented geodesics.
For a geodesic $\gamma=\gamma(t)$, set $\gamma_+$, $\gamma_- \in \partial \H^3$ as
\begin{equation}\label{eq:endpt}
  \gamma_+:=\lim_{t\rightarrow\infty}\gamma(t),\qquad
  \gamma_-:=\lim_{t\rightarrow-\infty}\gamma(t).
\end{equation}
Evidently, $\gamma_+$ and $\gamma_-$
are independent of choice of a representative of $[\gamma]$,
and $(\gamma_+,\gamma_-) \in (\partial \H^3\times\partial \H^3) \setminus \Delta$ holds,
where $\Delta$ is the diagonal set of $\partial \H^3\times\partial \H^3$.
Conversely, for any distinct points $a,~b \in \partial \H^3$,
there exists a unique equivalence class $[\gamma] \in \LH$ 
such that $\gamma_+=a$, $\gamma_-=b$.
Thus, we can identify 
$\LH = (\partial \H^3\times\partial \H^3) \setminus \Delta$ as a set.
Now we recall the {\it upper-half space model} of $\H^3$:
\begin{equation}\label{eq:upper}
  \R^3_{+}=\left( \left\{\left. (w,r)\in \C \times \R \,\right|\, r>0 \right\},\, 
  \frac{ dw d\bar{w}+dr^2}{r^2} \right).
\end{equation}
A map
\begin{equation}\label{eq:isometry} 
  \Psi : \H^3 \ni \left(
  \begin{array}{cc}
    x_0+x_3 & x_1+ix_2\\ 
    x_1-ix_2 & x_0-x_3 
  \end{array}\right)  
  \longmapsto \left( \frac{x_1+ix_2}{x_0-x_3}, \frac{1}{x_0-x_3}\right) \in\R^3_+ 
\end{equation}
gives an isometry.
The geodesics of $\R^3_+$ are divided into two types:
straight lines parallel to the $r$-axis and semicircles perpendicular to the $w$-plane.

Identifying $\partial \H^3$ with the Riemann sphere $\hat{\C}:=\C\cup\{\infty\}$, 
we may consider $\gamma_{+}$ and $\gamma_{-}$ as points in $\hat{\C}$.
Then we set an open subset $\mathcal{U}$ of $\LH$ as
\begin{equation}\label{eq:nbd}
  \mathcal{U}:=\left\{\left. [\gamma] \in \LH \,\right|\, 
  \gamma_+ \neq 0,\, \gamma_- \neq \infty \right\},
\end{equation}
and complex numbers $\m_1$, $\m_2$ as
\begin{equation}\label{eq:holcoord}
  \m_1:=-\gamma_-,\qquad
  \m_2:=\frac{1}{\bar{\gamma}_+}
\end{equation}
for $[\gamma] \in \mathcal{U}$ (see Figure \ref{fig:upper}).
Georgiou and Guilfoyle \cite{GG} proved that 
$(\mathcal{U}; (\m_1,\m_2))$ defines a local holomorphic coordinate system of $\LH$ 
compatible to the complex structure $J$, 
and the map $[\gamma] \longmapsto (\m_1,\m_2)$ extends to a biholomorphic map 
\[
  (\LH,J) \bihol (\hat{\C}\times\hat{\C})\setminus\hat{\Delta}, 
\]
where $\hat{\Delta}=\{ (\m_1, \m_2) \in \C^2 \,|\, 
1+\m_1\bar{\m}_2=0\} \cup \{ (0,\infty),\, (\infty,0) \}$,
so-called the reflected diagonal.

\begin{figure}[htb]
 \begin{tabular}{cc}
  \resizebox{11cm}{!}{\includegraphics{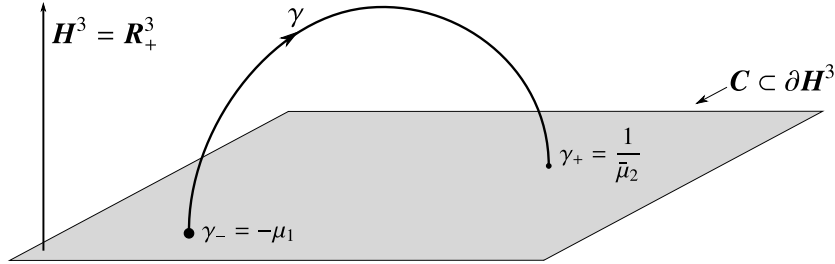}}
 \end{tabular}
 \caption{The holomorphic coordinate system $(\m_1,\m_2).$}
 \label{fig:upper}
\end{figure}

\begin{remark}[As a complex line bundle]
Over the complex projective line $\P^1$, the map 
\[
  \Pi : \LH \ni [\gamma] \longmapsto \gamma_- \in \P^1
\]
gives a complex line bundle. 
Each fiber of $\gamma_-$ is $\P^1\setminus\{\gamma_-\}$ which is identified with $\C$.
It is easy to see that $\Pi$ is a trivial bundle $\mathcal{O}_{\P^1}(0)$.
On the other hand, the space $\LR$ of oriented geodesics in the Euclidean $3$-space 
is biholomorphic to the holomorphic tangent bundle $T\P^1$ of $\P^1$ \cite{GK}.
That is $\LR \cong \mathcal{O}_{\P^1}(2)$.
This implies that $\LH$ is not isomorphic to $\LR$ as a line bundle over $\P^1$.
\end{remark}

\subsubsection{The invariant metrics, K\"ahler and para-K\"ahler structures}
\hspace{2mm}

The isometric action of $\SL(2,\C)$ on $\H^3$ as in \eqref{eq:isom} induces 
an action on $\partial\H^3=\hat{\C}$ as
\[
  \hat{\C} \ni z \longmapsto \frac{a_{11} z+a_{12} }{ a_{21} z+a_{22} } \in \hat{\C},
\]
where $a= (a_{ij}) \in \SL(2,\C)$.
This action induces a holomorphic and transitive action of 
${\rm Isom}_0(\H^3)=\PSL(2,\C)$ on 
$\LH=(\hat{\C}\times\hat{\C})\setminus\hat{\Delta}$ as
\begin{equation}\label{eq:hol-act}
  (\hat{\C}\times\hat{\C})\setminus\hat{\Delta} \ni (\m_1,\m_2) \longmapsto 
  \left( \frac{- a_{11} \m_1+a_{12} }{ \hphantom{-}a_{21} \m_1-a_{22} }, 
  \frac{\bar{a}_{22} \m_2+\bar{a}_{21} }{\bar{a}_{12} \m_2+\bar{a}_{11} } \right)
  \in (\hat{\C}\times\hat{\C})\setminus\hat{\Delta},
\end{equation}
for $a= (a_{ij}) \in \PSL(2,\C)$.
If we set a $\C$-valued symmetric $2$-tensor on $\LH$ as
\begin{equation}\label{eq:Killing-coord} 
  \mathcal{G}:= \frac{4\,d\mu_1d\bar{\mu}_2}{ (1+ \mu_1\bar{\mu}_2)^2 },
\end{equation}
then it holds that
\begin{equation}\label{eq:invariant-metric}
  \mathcal{G}_{\theta} 
  := \Re \left( e^{-i\theta} \mathcal{G} \right) 
  =  (\cos \theta)\, \mathcal{G}^{\mathfrak{r}} + (\sin \theta)\, \mathcal{G}^{\mathfrak{i}}
\end{equation}
defines a pseudo-Riemannian metric on $\LH$ 
of neutral signature for each $\theta \in \R/ 2 \pi \Z$,
which is invariant under the action given in \eqref{eq:hol-act},
where $\mathcal{G}^{\mathfrak{r}}$ and $\mathcal{G}^{\mathfrak{i}}$ are
the neutral metrics given by the real and imaginary part of $\mathcal{G}$, respectively,
\begin{equation}\label{eq:metrics}
  \mathcal{G}^{\mathfrak{r}}
  :=\dfrac{1}{2} \left\{
  \dfrac{ 4\,d\mu_1 d\bar{\mu}_2}{(1+\mu_1 \bar{\mu}_2)^2}+
  \dfrac{ 4\,d\mu_2 d\bar{\mu}_1}{(1+\mu_2 \bar{\mu}_1)^2} \right\},\quad
  \mathcal{G}^{\mathfrak{i}}
  :=\dfrac{1}{2i} \left\{ 
  \dfrac{ 4\,d\mu_1 d\bar{\mu}_2}{ (1+ \mu_1\bar{\mu}_2)^2 }-
  \dfrac{ 4\,d\mu_2 d\bar{\mu}_1}{ (1+ \mu_2\bar{\mu}_1)^2} \right\}.
\end{equation}
Conversely, Salvai \cite{Salvai} proved that 
any pseudo-Riemannian metric on $\LH$ 
invariant under the action as in \eqref{eq:hol-act}
is a constant multiple of $\mathcal{G}_{\theta}$ for some $\theta \in \R/ 2 \pi \Z$.
Thus we call $\mathcal{G}_{\theta}$ ($\theta \in \R/ 2 \pi \Z$) {\it invariant metrics\/}.
Any invariant metric $\mathcal{G}_{\theta}$ is K\"ahler 
with respect to the natural complex structure
\begin{equation}\label{eq:cpx}
  J \left( \frac{\partial}{\partial \m_1} \right) = i\frac{\partial}{\partial \m_1}, \qquad 
  J \left( \frac{\partial}{\partial \m_2} \right) = i\frac{\partial}{\partial \m_2}. 
\end{equation}
On the other hand, a involutive $(1,1)$-tensor $P$ on $\LH$ given as
\begin{equation}\label{eq:P-cpx}
  P \left( \frac{\partial}{\partial \m_1} \right) = -\frac{\partial}{\partial \m_1}, \qquad 
  P \left( \frac{\partial}{\partial \m_2} \right) = \frac{\partial}{\partial \m_2}
\end{equation}
is a {\it para-K\"ahler structure} on $\LH$ for any $\mathcal{G}_{\theta}$.
That is, for $[\gamma]$ in $\LH$, we have
\[
  \dim_{\R}\{ X \in T_{[\gamma]}\LH \,|\, P(X)=\pm X \}=2,\quad
  \mathcal{G}_{\theta}(P \cdot ,P \cdot )=-\mathcal{G}_{\theta}(\cdot, \cdot),\quad
  \nabla^{L} P=0,
\] 
where $\nabla^{L}$ is the common 
Levi-Civita connection of $(\LH, \mathcal{G}_{\theta})$ for all $\theta$.

\section{The Invariant Metrics and the Canonical Symplectic Form}
\label{sec:geomstr}

In this section, we shall characterize two neutral metrics 
$\mathcal{G}^{\mathfrak{r}}$ and $\mathcal{G}^{\mathfrak{i}}$ given 
in \eqref{eq:metrics}: 
both the para-K\"ahler form of $(\LH, \mathcal{G}^{\mathfrak{r}},P)$ and 
the K\"ahler form of $(\LH, \mathcal{G}^{\mathfrak{i}},J)$ coincide with 
the twice of the canonical symplectic form on $\LH$ up to sign 
(Proposition \ref{prop:symplectic}).
Moreover, identifying $\LH=\SL(2,\C)/\GL(1,\C)$, 
we prove that $\mathcal{G}$ in \eqref{eq:Killing-coord} coincides with 
the $\C$-valued symmetric $2$-tensor induced from the Killing form 
of the Lie algebra $\mathfrak{sl}(2,\C)$
of $\SL(2,\C)$ up to real constant multiplication (Proposition \ref{prop:BG}).

\subsubsection*{The canonical symplectic form}
\hspace{2mm}

Let $\omega$ be the {\it canonical symplectic form} on $\LH$, 
that is, $\omega$ is the symplectic form on $\LH$ satisfying
\begin{equation}\label{eq:cano-form}
  \hat{\pi}^* \omega = d \Theta,
\end{equation}
where $\Theta$ is the canonical contact form given in \eqref{eq:cano-cont} 
on the unit tangent bundle $U\H^3$,
and $\hat{\pi} : U\H^3 \rightarrow \LH$ is the projection as in \eqref{eq:pihat}.

We denote by $\omega_{J}$ the K\"ahler form of $(\LH, \mathcal{G}^{\mathfrak{i}}, J)$,
and by $\omega_{P}$ the para-K\"ahler form of $(\LH, \mathcal{G}^{\mathfrak{r}}, P)$,
that is,
\begin{equation}\label{eq:K-pK}
  \omega_{J}=\mathcal{G}^{\mathfrak{i}}(\cdot, J\cdot), \qquad 
  \omega_{P}=\mathcal{G}^{\mathfrak{r}}(\cdot, P\cdot). 
\end{equation}
Then we have the following

\begin{proposition}\label{prop:symplectic}
\[
  \omega_{J}=-\omega_{P}=2\omega.
\]
\end{proposition}

To prove this, we introduce metrics on $U\H^3$ and $\LH$ induced 
from the Killing form of $\mathfrak{sl}(2,\C)$ 
considering $U\H^3$ and $\LH$ as homogeneous spaces of $\SL(2,\C)$.

\subsubsection*{The Killing form of $\mathfrak{sl}(2,\C)$}
\hspace{2mm}

Let $B$ be the half of the Killing form of the Lie algebra 
$\mathfrak{sl}(2,\C)$ of $\SL(2,\C)$, i.e.,
\begin{equation}\label{eq:Killing}
  B(X,Y)=2\trace (XY), \qquad X,Y\in \mathfrak{sl}(2,\C).
\end{equation}
Then we set $B^{\mathfrak{r}}$ and $B^{\mathfrak{i}}$ 
to be the real and imaginary part of $B$, respectively: 
\begin{equation}\label{eq:reim-Killing}
 B^{\mathfrak{r}}:=\Re B,\qquad B^{\mathfrak{i}}:=\Im B.
\end{equation}

\begin{remark}
  The special linear group $\SL(2,\C)$ is the double cover 
  of the restricted Lorentz group $\SO^+(1,3)$.
  The Killing form of the real Lie algebra of $\mathfrak{so}(1,3)$ of $\SO^+(1,3)$ 
  coincides with a constant multiple of $B^{\mathfrak{r}}$.
\end{remark}

\subsubsection*{The unit tangent bundle}
\hspace{2mm}

The tangent space of the unit tangent bundle $U\H^3=\SL(2,\C)/\U(1)$ 
as in \eqref{eq:unitan} 
at $(\sigma_0,\sigma_3) \in U\H^3$ is identified with the orthogonal complement of 
the Lie algebra $\mathfrak{u}(1)$ of $\U(1)$ with respect to $B^{\mathfrak{r}}$, that is,
\begin{eqnarray*}
  T_{(\sigma_0,\sigma_3)} U\H^3
  =\mathfrak{u}(1)^{\perp}
  =\left\{\left. i \varepsilon \sigma_3+h_{\xi}+ v_{\eta} \,\right|\, 
  \varepsilon\in \R,\,\xi,\eta \in \C \right\},
\end{eqnarray*}
where $\sigma_0$, $\sigma_3$ are as in \eqref{eq:basis}, 
and $h_{\xi}$, $v_{\eta}$ are defined by
\begin{equation}\label{eq:horver}
  h_{\xi}=
  \left(\begin{array}{cccc}
  0 & \xi \\
  \bar{\xi} & 0
  \end{array}\right),\qquad
  v_{\eta}=
  \left(\begin{array}{cccc}
  0 & -\eta \\
  \bar{\eta} & 0
  \end{array}\right).
\end{equation}
These notations are used since  $h_{\xi}$, $v_{\eta}$ are 
horizontal and vertical tangent vectors of
the sphere bundle $\pi : U\H^3\rightarrow \H^3$ given in \eqref{eq:pi}, respectively.
The restriction of $B^{\mathfrak{r}}$ in \eqref{eq:reim-Killing} 
to $T_{(\sigma_0,\sigma_3)} U\H^3$ can be written by
\begin{equation}\label{eq:metric-on-UH3}
  B^{\mathfrak{r}}(X,X)=4(\varepsilon^2 + |\xi|^2-|\eta|^2), 
\end{equation}
for $X= i \varepsilon \sigma_3 + h_{\xi} + v_{\eta} \in T_{(\sigma_0,\sigma_3)} U\H^3$.
Thus $B^{\mathfrak{r}}$ defines a pseudo-Riemannian metric 
$B_{U}$ on $U\H^3$ of signature $(+,+,+,-,-)$.
Moreover, the projection
\begin{equation}\label{eq:adjust}
  \pi : (U\H^3, {B}_{U}) \longrightarrow (\H^3,\inner{~}{~})
\end{equation}
defined as in \eqref{eq:pi} is a pseudo-Riemannian submersion.

\subsubsection*{The space of oriented geodesics}
\hspace{2mm}

Consider the smooth and transitive action of $\SL(2,\C)$ given as
\[
  \LH \ni [\gamma] \longmapsto [a\gamma a^*] \in \LH, 
\]
for $a \in \SL(2,\C)$, where $[a\gamma a^*]$ is 
the equivalence class of the geodesic $a\gamma(t)a^{\ast}$ 
for some representative $\gamma$ of $[\gamma]$. 
Note that this action coincides with the action given in \eqref{eq:hol-act}.
If we denote by $\gamma_{\sigma_0,\sigma_3}$ the geodesic in $\H^3$ starting 
at $\sigma_0$ with initial velocity $\sigma_3$,
then the isotropy subgroup of $\SL(2,\C)$ at 
$[\gamma_0]:=[\gamma_{\sigma_0,\sigma_3}] \in \LH$ 
is given by
\[
  \left\{\left. \left(
  \begin{array}{cccc}
  \l & 0 \\
  0 & \l^{-1}
  \end{array}
  \right)
  \,\right|\, \l \in \C \setminus \{0\} \right\},
\]
which is identified with the general linear group $\GL(1,\C)$.
Hence we have 
\begin{equation}\label{eq:LH}
  \LH 
  =\SL(2,\C)/\GL(1,\C)
  =\left\{\left. [a \gamma_0 a^{\ast}] \,\right|\, a \in {\rm SL}(2,\C) \right\}.
\end{equation}
Then the tangent space of $\LH$ at $[\gamma_0]$ is identified with 
the orthogonal complement of the Lie algebra $\mathfrak{gl}(1,\C)$ 
of $\GL(1,\C)$ with respect to $B^{\mathfrak{r}}$, that is,
\[
  T_{[\gamma_0]}\LH 
  = \mathfrak{gl}(1,\C)^{\perp} 
  = \left\{\left. h_{\xi}+v_{\eta} \,\right|\, \xi,\eta \in \C \right\},
\]
where $h_{\xi}$ and $v_{\eta}$ are horizontal and vertical vectors 
of $T_{(\sigma_0,\sigma_3)}U\H^3$ defined in \eqref{eq:horver}.
The restrictions to $T_{[\gamma_0]} \LH$ of $B^{\mathfrak{r}}$ and $B^{\mathfrak{i}}$ 
defined in \eqref{eq:reim-Killing} can be written by
\[
  B^{\mathfrak{r}}\left( X,X\right)= 4(|\xi|^2-|\eta|^2),\qquad
  B^{\mathfrak{i}}\left( X,X\right) = 8 \Im(\xi \bar{\eta}),
\]
for $X = h_{\xi}+v_{\eta} \in T_{[\gamma_0]} \LH$, respectively.
Thus $B^{\mathfrak{r}}$ and $B^{\mathfrak{i}}$ define pseudo-Riemannian metrics 
$B^{\mathfrak{r}}_{L}$ and $B^{\mathfrak{i}}_{L}$ on $\LH$ 
of neutral signature, respectively.
Of course, the projection
\begin{equation}\label{eq:hadjust}
  \hat{\pi} : (U\H^3, B_{U}) \longrightarrow (\LH, B^{\mathfrak{r}}_{L})
\end{equation}
defined in \eqref{eq:pihat} is a pseudo-Riemannian submersion.

Let $B_{L} := B^{\mathfrak{r}}_L + i B^{\mathfrak{i}}_L$ 
be the $\C$-valued $2$-tensor on $\LH =\SL(2,\C)/\GL(1,\C)$ 
induced from $B$ in \eqref{eq:Killing}. Then we have the following

\begin{proposition}\label{prop:BG}
For the the $\C$-valued symmetric $2$-tensor $\mathcal{G}$ on $\LH$ 
defined in \eqref{eq:Killing-coord}, it follows that
\[
  \mathcal{G}=-B_{L}.
\]
\end{proposition}

\begin{proof}
It is enough to check the equality at 
$[\gamma_0]=[\gamma_{\sigma_0,\sigma_3}] \in \LH$ only.
For a sufficiently small neighborhood $\mathcal{R}$ of the origin $o \in \R^4$,
consider a map $\psi : \mathcal{R} \rightarrow \SL(2,\C)$ given by
\begin{equation}\label{eq:parametrization}
  \psi(u_1,u_2,v_1,v_2)=
  \left(\begin{array}{cc}
  1 & u_1-iv_2 + i u_2 - v_1\\
  u_1-i v_2 -iu_2+v_1 & 1+(u_1 - i v_2)^2 + (u_2+i v_1)^2 
  \end{array}\right).
\end{equation}
This map $\psi$ may be considered as a parametrization of $\LH=\SL(2,\C)/\GL(1,\C)$ 
around $\psi(o)=[\gamma_0]$. For $\xi,\, \eta\in \C$, set
\begin{equation}\label{eq:X}  
  \Vect{\vect{x}}_{\xi,\eta} :=
  (\Re\xi) \left.\frac{\partial}{\partial u_1}\right|_{o} +
  (\Im \xi) \left.\frac{\partial}{\partial u_2}\right|_{o} +
  (\Re\eta) \left.\frac{\partial}{\partial v_1}\right|_{o} + 
  (\Im \eta) \left.\frac{\partial}{\partial v_2}\right|_{o} \in T_{o}\mathcal{R},
\end{equation}
and $X:=\psi_{*}(\Vect{\vect{x}}_{\xi,\eta}) \in T_{[\gamma_0]} \LH$.
Then we have $X=h_{\xi}+v_{\eta}$, and
\begin{equation}\label{eq:norm}
  B^{\mathfrak{r}}_{L}\left( X, X\right)
  ={B^{\mathfrak{r}}} \left( X, X\right) 
  = 4(|\xi|^2-|\eta|^2),\qquad
  B^{\mathfrak{i}}_{L}\left( X, X\right)
  ={B^{\mathfrak{i}}} \left( X, X\right) 
  = 8 \Im(\xi \bar{\eta})
\end{equation}
at $[\gamma_0] \in \LH$, where $h_{\xi}$, $v_{\eta}$ are given in \eqref{eq:horver}.

On the other hand, set
$\hat{\psi}:=\pi_1\circ \psi : \mathcal{R} \rightarrow \LH$, 
where $\pi_1 : \SL(2,\C) \ni a \mapsto [a \gamma_0 a^*] \in \LH$.
The coordinates $(\m_1,\m_2)$ (see \eqref{eq:holcoord}) of 
$\hat{\psi}(u_1,u_2,v_1,v_2)$ can be calculated as
\[
  \m_1(u_1,u_2,v_1,v_2) 
  =  -\frac{(u_1+iu_2)-(v_1+iv_2)}{1+(u_1 - i v_2)^2 + (u_2+i v_1)^2},\quad
  \m_2(u_1,u_2,v_1,v_2) 
  =  (u_1+iu_2)+(v_1+iv_2).
\]
Then $\hat{X}:=\hat{\psi}_{*}(\Vect{\vect{x}}_{\xi,\eta}) \in T_{[\gamma_0]} \LH$ is given by
\begin{eqnarray*}
  \hat{X} 
  =(-\xi+\eta)\frac{\partial}{\partial \m_1}+(\xi+\eta)\frac{\partial}{\partial \m_2}+
         (-\bar{\xi}+\bar{\eta})\frac{\partial}{\partial \bar{\m}_1}+(\bar{\xi}
         +\bar{\eta})\frac{\partial}{\partial \bar{\m}_2}.
\end{eqnarray*}
By \eqref{eq:norm}, we have 
\[ 
  \mathcal{G}^{\mathfrak{r}}( \hat{X},\hat{X} )
  = -4(|\xi|^2-|\eta|^2)
  =- B^{\mathfrak{r}}_{L}\left( X,X\right),\qquad
  \mathcal{G}^{\mathfrak{i}}(\hat{X},\hat{X})
  = -8\Im(\xi\bar{\eta})
  =- B^{\mathfrak{i}}_{L}\left( X,X\right)
\]
at $[\gamma_0] \in \LH$,
where $\mathcal{G}^{\mathfrak{r}}$ and $\mathcal{G}^{\mathfrak{i}}$ 
are as in \eqref{eq:metrics}.
\end{proof}

\begin{proof}[{\bf Proof of Proposition $\ref{prop:symplectic}$}] 
\hspace{2mm}

By a similar calculation as in the proof of Proposition \ref{prop:BG}, 
the complex structure $J$ in \eqref{eq:cpx} and 
the para-complex structure $P$ in \eqref{eq:P-cpx} satisfy
\[
  J(h_{\xi}+v_{\eta})=h_{i\xi}+v_{i\eta},\qquad P(h_{\xi}+v_{\eta})=h_{\eta}+v_{\xi},
\]
for a tangent vector $h_{\xi}+v_{\eta}\in T_{[\gamma_0]}\LH$.
Thus by Proposition \ref{prop:BG}, the K\"ahler form $\omega_J$ 
and the para-K\"ahler form $\omega_P$ 
defined in \eqref{eq:K-pK} can be calculated as
\begin{equation}\label{eq:K-pK-forms}
  \omega_{P}(X,Y) = -\omega_{J}(X,Y)
  =-2\Re(\xi \bar{\delta} - \eta\bar{\beta}),
\end{equation}
where $X=h_{\xi}+v_{\eta}$, $Y=h_{\beta}+v_{\delta} \in T_{[\gamma_0]} \LH$.

To calculate the canonical symplectic form $\omega$ in \eqref{eq:cano-form}, set
$\tilde{\psi}:=\pi_2\circ \psi : \mathcal{R} \rightarrow U\H^3$, 
where $\psi$ is the map in \eqref{eq:parametrization} and
$\pi_2:\SL(2,\C) \ni a \mapsto (a a^*, a\sigma_3 a^*) \in U\H^3$.
Then the horizontal lifts of $X=h_{\xi}+v_{\eta}$, 
$Y=h_{\beta}+v_{\delta} \in T_{[\gamma_0]} \LH$ are given by
$\tilde{X} :=\tilde{\psi}_{*}(\Vect{\vect{x}}_{\xi,\eta})=( h_{\xi}, h_{\eta} )$, 
$\tilde{Y}:=\tilde{\psi}_{*}(\Vect{\vect{x}}_{\beta,\delta})=( h_{\beta}, h_{\delta}) 
\in T_{(\sigma_0,\sigma_3)}U\H^3$,
where $h_{\xi}$, $h_{\beta}$, $\cdots$ are as in \eqref{eq:tangUH3} 
and $\Vect{\vect{x}}_{\xi,\eta}$, $\Vect{\vect{x}}_{\beta,\delta}$ are given in \eqref{eq:X}.
By \eqref{eq:K-pK-forms}, we have
\begin{eqnarray*}\label{eq:calc-cano}
  2\omega_{[\gamma_0]}(\tilde{X},\tilde{Y})
  &=& 2d\Theta_{(\sigma_0,\sigma_3)}(\tilde{X},\tilde{Y})
  = \inner{h_{\xi}}{h_{\delta} }-\inner{h_{\beta}}{h_{\eta}}\\
  &=& 2\Re(\xi \bar{\delta} - \eta\bar{\beta})
  = -\omega_{P}(X,Y) = \omega_{J}(X,Y)
\end{eqnarray*}
at $[\gamma_0] \in \LH$, where $\Theta$ denotes 
the canonical contact form in \eqref{eq:cano-cont}.
\end{proof}

\begin{remark}
\label{rem:matome}
The metric $\mathcal{G}^{\mathfrak{i}}=\Im \mathcal{G}$ in \eqref{eq:metrics} 
is the twice of the K\"ahler metric defined in \cite[Definition 12]{GG}.
In fact, we defined $\mathcal{G}$ as in \eqref{eq:Killing-coord} 
so that the double fibration 

\vspace{3mm}
\unitlength 0.1in
\begin{picture}( 38.3000,  4.2000)(-0.0000,-31.1000)
\put(27.0000,-32.7000){\makebox(0,0)[lb]{$(\LH=\SL(2,\C)/\GL(1,\C),\, 
B^{\mathfrak{r}}_L=-\mathcal{G}^{\mathfrak{r}})$}}%
\put(15.3000,-28.5000){\makebox(0,0)[lb]{$(U\H^3=\SL(2,\C)/\U(1),\, B_U)$}}%
\put(1.3000,-32.7000){\makebox(0,0)[lb]{$(\H^3=\SL(2,\C)/\SU(2),\, 
                                                                                             \langle~,~\rangle)$}}%
\put(11.9000,-29.0000){\makebox(0,0)[lb]{$\pi$}}%
\put(36.0000,-29.0000){\makebox(0,0)[lb]{$\hat{\pi}$}}%
{\color[named]{Black}{%
\special{pn 8}%
\special{pa 1470 2860}%
\special{pa 1090 3070}%
\special{fp}%
\special{sh 1}%
\special{pa 1090 3070}%
\special{pa 1158 3056}%
\special{pa 1138 3044}%
\special{pa 1140 3020}%
\special{pa 1090 3070}%
\special{fp}%
}}%
{\color[named]{Black}{%
\special{pn 8}%
\special{pa 3380 2850}%
\special{pa 3760 3060}%
\special{fp}%
\special{sh 1}%
\special{pa 3760 3060}%
\special{pa 3712 3010}%
\special{pa 3714 3034}%
\special{pa 3692 3046}%
\special{pa 3760 3060}%
\special{fp}%
}}%
\end{picture}
\vspace{6mm}

\noindent
is compatible, that is, both $\pi$ in \eqref{eq:adjust} 
and $\hat{\pi}$ in \eqref{eq:hadjust} are pseudo-Riemannian submersions.
\end{remark}

\begin{remark}[A relationship to the Fubini-Study metric]
\label{rem:signature}
Consider a holomorphic curve
$F : \P^1=\hat{\C} \rightarrow \LH$ given by
$F|_{\C} : \C \ni \m \longmapsto (\m,\m) \in \LH$.
The image of $F$ in $\LH$ can be considered as
\[
  \LoH=\left\{\left. [\gamma] \in\LH \,\right|\, 
     \gamma ~\text{through the origin}~ o=(0,0,0) \in \B^3\right\},
\]
where $\B^3$ denotes the {\it Poincar\'e ball model\/} of $\H^3$:
\[
  \B^3=\left(\left\{\left. (x,y,z)\in \R^3 \,\right|\, x^2+y^2+z^2<1 \right\},\, 
  4\frac{dx^2+dy^2+dz^2}{(1-x^2-y^2-z^2)^2} \right).
\]

\begin{figure}[htb]
 \begin{tabular}{cc}
  \resizebox{5cm}{!}{\includegraphics{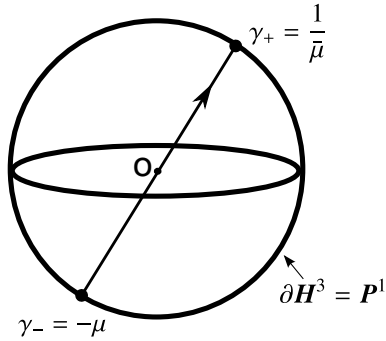}} 
 \end{tabular}
 \caption{An oriented geodesic through the origin.}
 \label{fig:Line}
\end{figure}

\noindent
We call $F$ or $\LoH$ the {\it standard embedding\/} of $\P^1$.
Moreover, if we equip on $\P^1$ the Fubini-Study metric 
$g_{FS}$ of constant curvature $1$, 
then the standard embedding
\[
  F : (\P^1,g_{FS}) \longrightarrow (\LH,\mathcal{G}^{\mathfrak{r}})
\]
is an isometric embedding.
In fact, we defined $\mathcal{G}$ as the opposite sign of $B_{L}$ 
(Proposition \ref{prop:BG}) because of this fact.
\end{remark}

\section{A Representation Formula for Developable Surfaces}
\label{sec:null}
In this section, we shall prove Theorem~\ref{thm:null} in the introduction.
First, we review fundamental facts on isometric immersions of $\H^2$ 
into $\H^3$ as surfaces in $\H^3$, 
and prove that isometric immersions of $\H^2$ into $\H^3$ are 
developable (Proposition \ref{prop:Portnoy}).
Then we shall prove Theorem~\ref{thm:null} (Proposition~\ref{prop:representation}).

\subsection{Isometric immersions and developable surfaces}
\hspace{2mm}

In this paper, a {\it surface} in $\H^3$ is considered as an immersion $f$ 
of a differentiable $2$-manifold $\Sigma$ into $\H^3$ (cf.\ \eqref{eq:Herm-model}):
\[
  f : \Sigma \longrightarrow \H^3 \subset \L^4 = \Herm(2). 
\]
We denote by $g=f^{\ast}\inner{~}{~}$ {\it the first fundamental form} of $f$.
For the unit normal vector field $\vect{\nu}$ of $f$, we denote by  $A$ and $I\!I$ 
the {\it shape operator} and the {\it second fundamental form} of $f$, respectively, that is,
$A = -(f_{\ast})^{-1}\circ \vect{\nu}_{\ast}$,
$I\!I(V,W)= - \inner{ \vect{\nu}_{\ast}(V) }{ f_{\ast}(W) }$,
where $V$ and $W$ are vector fields on $\Sigma$.
Let $k_1$, $k_2$ be the {\it principal curvatures} of $f$, then
the {\it extrinsic curvature} $K_{\rm ext}$ and the {\it mean curvature} $H$ 
can be written as
\[
  K_{\rm ext}=k_1k_2,\qquad H=\frac{k_1+k_2}{2},
\]
respectively.
If we denote by $K$ and $\nabla$ the Gaussian curvature and 
the Levi-Civita connection of the Riemannian 2-manifold $(\Sigma, g)$, respectively, 
then we have 
\begin{equation}\label{eq:Gauss}
  K=-1+K_{\rm ext},
\end{equation}
\begin{equation}\label{eq:Codazzi}
  \nabla_{V}A(W)=\nabla_{W}A(V),
\end{equation}
for vector fields $V,\, W$ on $\Sigma$.
We call \eqref{eq:Gauss} the {\it Gauss equation}, 
and \eqref{eq:Codazzi} the {\it Codazzi equation}.
A surface in $\H^3$ is said to be {\it extrinsically flat}
if its extrinsic curvature is identically zero.
By the Gauss equation, we have that an isometric immersion of $\H^2$ into $\H^3$
is a complete extrinsically flat surface.

On the other hand, any unit speed geodesic in $\H^3$ can be expressed as
\[
  \gamma_{p,v}(t)=p\cosh t+v\sinh t ,\qquad (p,v) \in U\H^3.
\]
\begin{definition}[Ruled surfaces and developable surfaces]
A  {\it ruled surface} in $\H^3$ is a locus of $1$-parameter family of geodesics in $\H^3$.
For a ruled surface $f: \Sigma \rightarrow \H^3$, 
there exists a local coordinate system $\varphi=(s,t)$ of $\Sigma$ such that
\[
  (f\circ\varphi^{-1})(s,t) = c(s) \cosh t + v(s) \sinh t ,
\]
where $c$ is a curve in $\H^3$ and 
$v$ is a unit normal vector field along $c$.
A ruled surface is said to be {\it developable} if it is extrinsically flat.
\end{definition}

Then we have the following
\begin{proposition}[{\cite[Theorem 4]{Portnoy}}]\label{prop:Portnoy}
A complete extrinsically flat surface in $\H^3$ is developable.
\end{proposition}
To show this, we first prove an analogue of 
{\it Massey's lemma} \cite[Lemma 2]{Massey} (cf.\ Remark \ref{rem:Massey}).
For a surface $f: \Sigma \rightarrow \H^3$,
a curve in $\Sigma$ is said to be {\it asymptotic} 
if each tangent space of the curve gives the kernel of 
the second fundamental form of $f$.

\begin{lemma}[Hyperbolic Massey's lemma]
\label{lem:Massey}
For an extrinsically flat surface $f : \Sigma \rightarrow \H^3$,
let $\mathcal{W}$ be the set of umbilic points of $f$
and $\gamma$ an asymptotic curve in the non umbilic point set 
$\mathcal{W}^c= \Sigma \setminus \mathcal{W}$.
Then the mean curvature $H$ of $f$ satisfies
\[
  \frac{\partial^2}{\partial t^2} \left( \frac{1}{H} \right) = \frac{1}{H}, 
\]
on $\gamma$, where $t$ denotes the arc length parameter of $\gamma$.
\end{lemma}

\begin{proof}
Take a non umbilic point $p \in \mathcal{W}^c$, 
and curvature line coordinate system $(s,v)$ around $p$
with $v$-curves asymptotic.
Then the first and second fundamental forms $g$ and $I\!I$ are expressed as
$g = g_{11} ds^2 + g_{22} dv^2 ,~ I\!I =h_{11}ds^2 ~(h_{11} \neq 0)$,
and hence the Codazzi equation~\eqref{eq:Codazzi} is equivalent to
\begin{equation}\label{0}
  \frac{\partial h_{11}}{\partial v}=\frac{h_{11}}{2g_{11}} \frac{\partial g_{11}}{\partial v} ,
\end{equation}
\begin{equation}\label{1}
  0=\frac{h_{11}}{2g_{11}} \frac{\partial g_{22}}{\partial s}. 
\end{equation}
By \eqref{1}, $g_{22}$ depends only on $v$.
Reparametrizing with $dt=\sqrt{g_{22}(v)}\, dv$, we obtain
$g = g_{11} ds^2 + dt^2,~ I\!I = h_{11} ds^2 ~(h_{11} \neq 0)$.
In this coordinate system, 
each $t$-curve is an asymptotic curve parametrized by arc length
and the Gaussian curvature $K$ of $f$ is written as
\[
  K=-\frac{1}{\sqrt{ g_{11} } }\frac{\partial^2 \sqrt{g_{11}}}{\partial t^2}.
\]
Since $f$ is extrinsically flat, the Gauss equation \eqref{eq:Gauss} yields
\begin{equation}
  \frac{\partial^2 \sqrt{g_{11}}}{\partial t^2}=\sqrt{g_{11}}. \label{2}
\end{equation}
On the other hand, by \eqref{0}, we have
\[
  \frac{\partial}{\partial t}\log\frac{h_{11}}{\sqrt{g_{11}}}
  =\frac{1}{h_{11}}\frac{\partial h_{11}}{\partial t}-\frac{1}{2 
        g_{11}}\frac{\partial g_{11}}{\partial t}=0,
\]
and hence there exists a function ${a}={a}(s)$ such that
\[
  h_{11}(s,t) = a(s) \sqrt{g_{11}(s,t)} \qquad (a(s) \neq 0). 
\]
Then the mean curvature $H$ of $f$ can be written as
$H =a(s)/(2\sqrt{g_{11}})$. Besides \eqref{2}, we have
\begin{equation*}
  \frac{\partial^2}{\partial t^2} \left( \frac{1}{H} \right)
  =\frac{\partial^2}{\partial t^2}\frac{2\sqrt{g_{11}}}{a(s)}
  =\frac{2}{a(s)}\frac{\partial^2}{\partial t^2}\sqrt{g_{11}}
  =\frac{2}{a(s)}\sqrt{g_{11}}
  =\frac{1}{H}.
\end{equation*}
\end{proof}

\begin{remark}
\label{rem:Massey}
Although original Massey's lemma \cite[Lemma 2]{Massey} is for flat surfaces in $\R^3$,
we can generalize it for extrinsically flat surfaces in $\S^3$ in the same way.
On the other hand, Murata and Umehara generalized Massey's lemma 
for a class of flat surfaces with singlarities ({\it flat fronts\/}) 
in $\R^3$ \cite[Lemma 1.15]{MurataUmehara}.
\end{remark}

\noindent
{\bf Proof of Proposition $\ref{prop:Portnoy}$}
\hspace{2mm}

Most part of this proof is a modification of the proof of 
Hartman-Nirenberg theorem given by Massey \cite{Massey}.
However, some part of the original Massey's proof is not valid for hyperbolic case, 
thus the final part of this proof is written carefully (see Claim below).

Let $f : \Sigma \rightarrow \H^3$ be a complete extrinsically flat surface
and $\mathcal{W}$ the set of umbilic points of $f$.
Since the restriction of $f$ to $\mathcal{W}$ is a totally geodesic embedding, 
$f|_{\mathcal{W}}$ is ruled.
By the proof of Lemma~\ref{lem:Massey}, 
for any non umbilic point in $\mathcal{W}^c=\Sigma\setminus\mathcal{W}$, 
there exists a local coordinate neighborhood $\left(U ; (s,t) \right)$ 
around the point such that
\[
  g = g_{11} ds^2 + dt^2 ,\qquad I\!I = h_{11} ds^2 \quad ( h_{11} \neq 0). 
\]
Then it can be shown that the geodesic curvature of each $t$-curve vanishes anywhere.
This means that any asymptotic curve in $\mathcal{W}^c$ is a part of 
geodesic in $\H^3$. For a fixed point $q \in \mathcal{W}^c$, 
let $G(q)$ be the unique asymptotic curve in $\mathcal{W}^c$ passing through $q$.
By Lemma~\ref{lem:Massey}, it follows that the mean curvature $H$ is given by
\begin{equation}\label{eq:meancurva}
  H=\frac{1}{a \cosh t + b \sinh t} 
\end{equation}
on $G(q)$, where $a,\, b$ are constants and $t$ denotes the distance induced from the 
first fundamental form of $f$ measured from $q$.
If $G(q)$ intersects with the boundary $\partial \mathcal{W}$,
the mean curvature $H$ vanishes at $Q \in \partial \mathcal{W} \cap G(q)$, 
a contradiction.
Thus any asymptotic curve in $\mathcal{W}^c$ does not intersect with 
the boundary of $\mathcal{W}^c$,
and hence we have $f|_{\mathcal{W}^c}$ is ruled.
It is sufficient to show the following

\vspace{2mm}
\noindent
{\bf Claim .}~{\it
$\partial \mathcal{W}$ is a disjoint union of geodesics in $\H^3$.
}

\begin{proof}
For a point $p \in \partial \mathcal{W}$,
there exists a sequence $\{p_n\}_{n \in \N}$ in $\mathcal{W}^c$ such that
$ \lim_{n\rightarrow \infty} p_n = p .$
Let $G(p_n)$ be the unique asymptotic curve through $p_n \in \mathcal{W}^c$.
Since $G(p_n)$ is a geodesic in $H^3$, we can express as
$G(p_n)(t)=p_n \cosh t + v_n \sinh t$,
with a unit tangent vector $v_n \in T_{p_n}\H^3$.
We shall prove that there exists $v$ of the limit of $\{v_n\}_{n \in \N}$, 
taking a subsequence, if necessary.
Set $p_n=(p_{0_n},\vect{p}_n)$, 
$v_n=(v_{0_n},\vect{v}_n) \in \L^4= \R \times \R^3 $.
Then we have
\[
  -p_{0_n}^2+|\vect{p}_n|_{E}^2=-1 ,\qquad  
  -v_{0_n}^2+|\vect{v}_n|_{E}^2=1,\qquad 
  -p_{0_n}v_{0_n}+\inner{\vect{p}_n}{\vect{v}_n}_{E}=0,
\]
for all $n \in \N$, where $\langle  \cdot,\cdot \rangle_{E}$ is 
the Euclidean inner product of $\R^3$ and
$|\cdot|_{E}$ is the associated Euclidean norm.
By the Cauchy-Schwartz inequality, 
\begin{equation*}
  |v_{0_n}|
  = \frac{1}{p_{0_n}}|\inner{\vect{p}_n}{\vect{v}_n}_{E}| 
  \leq \frac{1}{p_{0_n}}|\vect{p}_n|_{E}|\vect{v}_n|_{E} 
  = \sqrt{ \frac{p_{0_n}^2-1}{p_{0_n}^2} }\sqrt{v_{0_n}^2+1},
\end{equation*}
and we have
\begin{equation}
  \label{ineq:principle}
  \frac{|v_{0_n}|}{\sqrt{v_{0_n}^2+1}} \leq  \sqrt{ 1-\frac{1}{p_{0_n}^2} } \leq 1, 
\end{equation}
for $n \in \N$. If $|v_{0_n}| \rightarrow \infty$,
\[
  \frac{|v_{0_n}|}{\sqrt{v_{0_n}^2+1}} \longrightarrow 1
\]
holds and we have
$p_{0_n} \rightarrow \infty$ by \eqref{ineq:principle}.
But it contradicts with $\lim_{n\rightarrow \infty} p_n = p$.
Thus there exists $R>0$ such that $\{v_n\}_{n \in \N} \subset B(R)$, 
where $B(R)=\{ {}^t(x_0,x_1,x_2,x_3)\in \L^4 \,|\, x_0^2+x_1^2+x_2^2+x_3^2 \leq R \}$.
If we set $\S^3_1:=\{ \vect{x}\in \L^4 \,|\, \inner{\vect{x}}{\vect{x}}=1 \}$,
we also have $\{v_n\}_{n \in \N} \subset \S^3_1 \cap B(R)$.
Since $\S^3_1 \cap B(R)$ is compact, 
there exists a subsequence $\{ v_{n_k} \} \subset  \{ v_n \}$ 
such that $\lim_{k\rightarrow \infty} v_{n_k} = v$ exists.
Therefore we can define
$G(p) = \lim_{n \rightarrow \infty} G(p_n) \subset \mathcal{W}^c \cup \partial \mathcal{W}$
as $\gamma_{p,v}$. If $G(p) \cap \mathcal{W}^c$ is non empty,
take $q \in G(p) \cap \mathcal{W}^c$.
Then $G(q)=G(p)$ and hence $G(q)$ through $p\in \partial \mathcal{W}$, 
a contradiction.
Thus $G(p) \subset \partial \mathcal{W}$.
\end{proof}

As a corollary, we have the following
\begin{corollary}\label{fact:dev}
 An isometric immersion of $\H^2$ into $\H^3$ is a complete 
 developable surface in $\H^3$.
\end{corollary}

\subsection{Proof of Theorem \ref{thm:null}}
\hspace{2mm}

Since a ruled surface in $\H^3$ is a locus of 1-parameter family of geodesics,
it gives a curve in the space of oriented geodesics $\LH$.
Conversely, a curve in $\LH$ generates a ruled surface 
(it may have singularities) in $\H^3$.
Here, we shall investigate the curves given by developable surfaces in $\H^3$.
Let $(\m_1,\m_2)$ be a point in $\LH$ as in \eqref{eq:holcoord}.
Then it corresponds to a equivalence class $[\gamma]$, 
where $\gamma(t)$ is expressed as
\begin{equation}\label{eq:geodesic}
  \gamma(t)=\frac{1}{\left| 1+\m_1\bar{\m}_2 \right| }
  \left(
  \begin{array}{cc}
  e^{t}+e^{-t}|\mu_1|^2&
  e^t \mu_2 -e^{-t} \mu_1  \\
  e^t \bar{\mu}_2 -e^{-t} \bar{\mu}_1 &
  e^{t}|\mu_2|^2 + e^{-t}
  \end{array}
  \right) \in \Herm(2).
\end{equation}

A regular curve in a pseudo-Riemannian manifold 
is called {\it null\/} (resp.\ {\it causal\/}) if every tangent vector gives 
null (resp.\ timelike or null) direction.
Recall that the neutral metrics $\mathcal{G}^{\mathfrak{r}}$ 
and $\mathcal{G}^{\mathfrak{i}}$ are defined in \eqref{eq:metrics}.
Theorem~\ref{thm:null} is a direct conclusion of the following

\begin{proposition}\label{prop:representation}
For a regular curve $\alpha(s)=(\m_1(s),\m_2(s)) : 
\R \supset I \rightarrow \mathcal{U} \subset \LH$
which is null with respect to $\mathcal{G}^{\mathfrak{i}}$ and 
causal with respect to $\mathcal{G}^{\mathfrak{r}}$,
a map $f : I \times \R \rightarrow \H^3$ defined by
\begin{equation}\label{eq:representation}
  f(s,t)=
  \frac{1}{\left| 1+\m_1(s)\bar{\m}_2(s) \right| }
  \left(
  \begin{array}{cc}
  e^{t}+e^{-t}|\mu_1(s)|^2&
  e^t \mu_2(s) -e^{-t} \mu_1(s)\\
  e^t \bar{\mu}_2(s) -e^{-t} \bar{\mu}_1(s) &
  e^{t}|\mu_2(s)|^2 + e^{-t}
  \end{array}
 \right)
\end{equation}
is a developable surface.
Conversely, any developable surface generated by 
complete geodesics in $\H^3$ can be written locally in this manner.
\end{proposition}

\begin{proof}
By \eqref{eq:geodesic}, a parametrization of the locus of $\alpha$ can be written by 
$f $ as in \eqref{eq:representation}.
First we shall prove that if $\alpha$ is null with respect to $\mathcal{G}^{\mathfrak{i}}$ 
and causal with respect to $\mathcal{G}^{\mathfrak{r}}$, then $f$ is an immersion.
Set 
\begin{equation}\label{eq:regularity}
  \Lambda(s,t)
  :=| f_s \times f_t |^2
  = \frac{e^{2t} |\m'_2|^2 + e^{-2t} |\m'_1|^2 }{\left| 1+\m_1\bar{\m}_2 \right|^2}
  -\frac{1}{2}\mathcal{G}^{\mathfrak{r}}(\alpha',\alpha'),
\end{equation}
where ${~}'=d/ds$, $f_s=\partial f/\partial s$, $f_t=\partial f/\partial t$
and $\times$ denotes the cross product of $\H^3$ as in \eqref{eq:cross}.
Thus we have $\Lambda(s,t)$ is positive
if $\mathcal{G}^{\mathfrak{r}}(\alpha',\alpha')$ is negative.
Consider the case $\mathcal{G}^{\mathfrak{r}}(\alpha',\alpha')=0$ at $s \in I$.
Since $\alpha$ is null with respect to $\mathcal{G}^{\mathfrak{i}}$, we have
$|\m'_1||\m'_2|=0$.
The regularity of $\alpha$ shows that either $\m'_1=0$ or $\m'_2=0$ occurs.
Without loss of generality, we may assume $\m'_1=0$.
Then the regularity of $\alpha$ means $\m'_2 \neq 0$, and then
$\Lambda(s,t)=e^{2t} |\m'_2|^2/\left| 1+\m_1\bar{\m}_2 \right|^2$
is positive. Thus $f$ is an immersion.

Next we shall show that $f$ is extrinsically flat.
The unit normal vector field $\vect{\nu}$ of $f$ is given by
\begin{equation}\label{eq:unitnormal}
  \vect{\nu}(s,t)
  =\frac{ f_s \times f_t }{| f_s \times f_t |}
  =\frac{i}{|1+\m_1\bar{\m}_2|^3 \sqrt{\Lambda(s,t)} }
  \left(
  \begin{array}{cc}
  a(s,t) & z(s,t)\\ -\bar{z}(s,t) & b(s,t)
  \end{array}
  \right),
\end{equation}
where
\[
  a(s,t) = 2i\Im \{ e^{t} (1+\m_1\bar{\m}_2) \bar{\m}_1 \m'_2 
  - e^{-t} (1+\m_2\bar{\m}_1) \bar{\m}_1 \m'_1\},
\]
\[
  b(s,t) =-2i\Im \{ e^{t} (1+\m_1\bar{\m}_2) \bar{\m}_2 \m'_2 
  - e^{-t} (1+\m_2\bar{\m}_1) \bar{\m}_2 \m'_1\},
\]
\[
  z(s,t) 
  = -e^{t}\{ (1+\m_1\bar{\m}_2) \m'_2
           + (1+\m_2\bar{\m}_1) \m_1\m_2\bar{\m}'_2 \}
           +e^{-t}\{ (1+\m_2\bar{\m}_1) \m'_1
           + (1+\m_1\bar{\m}_2) \m_1\m_2\bar{\m}'_1 \}.
\] 
Since
\[
  K_{\rm ext} 
  =\frac{\inner{f_s}{\vect{\nu}_s}\inner{f_t}{\vect{\nu}_t}
      -\inner{f_s}{\vect{\nu}_t}\inner{f_t}{\vect{\nu}_s}}{
       \inner{f_s}{f_s}\inner{f_t}{f_t}-\inner{f_s}{f_t}^2}
  \qquad \text{and} \qquad
  \mathcal{G}^{\mathfrak{i}}(\alpha', \alpha') 
  = \Im\frac{4\m'_1\bar{\m}'_2}{(1+\m_1\bar{\m}_2)^2},
\]
we have
\begin{equation}\label{eq:Kext}
  K_{\rm ext}
  =
  \frac{i}{ \sqrt{\Lambda(s,t)}^3 }
  \left\{\frac{\m'_1\bar{\m}'_2}{
  (1+\m_1\bar{\m}_2)^2}-\frac{\m'_2\bar{\m}'_1}{(1+\m_2\bar{\m}_1)^2} \right\}
  =
  \frac{-1}{2 \sqrt{\Lambda(s,t)}^3 }
  \mathcal{G}^{\mathfrak{i}}(\alpha',\alpha').
\end{equation}
Therefore $\mathcal{G}^{\mathfrak{i}}(\alpha',\alpha') =0$ 
if and only if $K_{\rm ext} =0$.

Conversely, for a ruled surface $\hat{f} : \Sigma \rightarrow \H^3$,
there exists a $1$-parameter family $\alpha=\alpha(s)$ of geodesics
such that its locus coincides with the given surface $\hat{f}$.
Using a suitable isometry, 
we may assume that the image of $\alpha$ is included 
in $\mathcal{U}$ in \eqref{eq:nbd}, that is,
\[
  \alpha :  \R \supset I \ni s \longmapsto  (\m_1(s),\m_2(s)) \in \mathcal{U} \subset \LH.
\]
Thus $\hat{f}$ is given by $f$ as in \eqref{eq:representation} locally.
We shall prove that, if the ruled surface $\hat{f}$ is developable, 
$\alpha$ is a regular curve
which is null with respect to $\mathcal{G}^{\mathfrak{i}}$ and 
causal with respect to $\mathcal{G}^{\mathfrak{r}}$.
If there exists a point such that $\alpha'=0$, $\hat{f}$ 
is not an immersion because of \eqref{eq:regularity}.
Thus $\alpha$ is a regular curve.
Moreover $\alpha$ is a null with respect to $\mathcal{G}^{\mathfrak{i}}$ 
by \eqref{eq:Kext}.
Then we shall prove $\alpha$ is causal with respect to $\mathcal{G}^{\mathfrak{r}}$.
If $\mathcal{G}^{\mathfrak{r}}(\alpha',\alpha')>0$,
\[
  \mathcal{G}^{\mathfrak{r}}(\alpha',\alpha')
  =\Re\frac{4\m'_1\bar{\m}'_2}{(1+\m_1\bar{\m}_2)^2}
  =\frac{4|\m'_1||\m'_2|}{|1+\m_1\bar{\m}_2|^2},
\]
holds since $\mathcal{G}^{\mathfrak{i}}(\alpha',\alpha')=0$.
Then we have
\begin{eqnarray*}
  \Lambda(s,t)=
  \frac{4|\m'_1||\m'_2|}{\left| 1+\m_1\bar{\m}_2 \right|^2}
  \sinh^2 \left( t+\frac{1}{2}\log\frac{|\m'_2|}{|\m'_1|} \right),
\end{eqnarray*}
and hence $\hat{f}$ has a singular point at $t=(\log|\m'_1|-\log|\m'_2|)/2$, 
a contradiction.
\end{proof}

\subsection{Examples}
\hspace{2mm}

Nomizu \cite{Nomizu} constructed fundamental examples of 
complete developable surfaces in $\H^3$ 
(cf. Figure~\ref{fig:Nomizu} in the introduction).

\begin{example}[Hyperbolic $2$-cylinders, {\cite[Example 1]{Nomizu}}]
\label{ex:1}
Let $\D$ be the unit disc in $\C$.
For a regular curve $\zeta(s) : \R \rightarrow \D$, set
\[
  \alpha_1(s)=(-\zeta(s),\zeta(s)). 
\]
Then $\alpha_1$ determines a regular curve in 
$\LH=(\hat{\C}\times\hat{\C})\setminus\hat{\Delta}$,
which is null with respect to $\mathcal{G}^{\mathfrak{i}}$ 
and causal with respect to $\mathcal{G}^{\mathfrak{r}}$.
Thus by Theorem \ref{thm:null}, the locus of $\alpha_1$ is a developable surface,
called {\it hyperbolic $2$-cylinder\/}.
Figure~\ref{fig:Nomizu} (B) shows an example of $\zeta(s)=e^{is}/3$.
\end{example}

\begin{example}[Ideal cones, {\cite[Example 2]{Nomizu}}]
\label{ex:2}
For a regular curve $\m(s) : \R \rightarrow \C$,
set
\[
  \alpha_2(s)=(\m(s),0). 
\]
Then $\alpha_2$ determines a regular curve in 
$\LH=(\hat{\C}\times\hat{\C})\setminus\hat{\Delta}$,
which is null with respect to both $\mathcal{G}^{\mathfrak{i}}$ and 
$\mathcal{G}^{\mathfrak{r}}$.
Thus by Theorem \ref{thm:null}, the locus of $\alpha_2$ is a developable surface.
Figure~\ref{fig:Nomizu} (C) shows an example of $\m(s)=e^{is}/2$.
We will see this example more precisely in Section \ref{sec:exponential}.
\end{example}

\begin{example}[Rectifying developables of helices, {\cite[Example 3]{Nomizu}}]
\label{ex:3}
For constants $\kappa,\, \tau \in \R \setminus \{0\}$, set
$a_{\pm}:=\sqrt{(\kappa \pm 1)^2+\tau^2}$,
$A_{\pm}:=\sqrt{\pm(1-\kappa^2-\tau^2)+a_{+}a_{-} }$
and $\alpha_3 : \R \rightarrow \C^2$ as
\begin{multline*}
  \alpha_3(s)= \left(
  \kappa \frac{4\sqrt{2}\sqrt{\kappa^2+\tau^2}i+4\tau A_{-}}{
  (\sqrt{2}\sqrt{\kappa^2+\tau^2}i+4\tau A_{+})( a_{+} + a_{-} )^2+4\kappa A_{-}}
  \exp \left( \frac{ A_{+}+iA_{-}  }{ \sqrt{2} }s \right) ,\right.\\
  \left.
  \frac{1}{\kappa}
  \frac{(\sqrt{2}\sqrt{\kappa^2+\tau^2}-\tau A_{+})( a_{+} + a_{-} )^2-4\kappa A_{-}}{
  4\sqrt{2}\sqrt{\kappa^2+\tau^2}i+4\tau A_{-}- ( a_{+} + a_{-} )^2 A_{+}}
  \exp \left( \frac{ -A_{+}+iA_{-}  }{ \sqrt{2} }s \right)
  \right).
\end{multline*}
Then $\alpha_3$ determines a regular curve in 
$\LH=(\hat{\C}\times\hat{\C})\setminus\hat{\Delta}$,
which is null with respect to $\mathcal{G}^{\mathfrak{i}}$ 
and causal with respect to $\mathcal{G}^{\mathfrak{r}}$.
Thus by Theorem \ref{thm:null}, the locus of $\alpha_3$ is a developable surface.
In fact, this is a rectifying developable \cite{Nomizu}
of the helix of constant curvature $\kappa$ and torsion $\tau$ in $\H^3$.
Figure~\ref{fig:Nomizu} (D) shows an example of $\kappa=\tau=1$.
\end{example}

\section{Ideal Cones and Behavior of the Mean Curvature}
\label{sec:exponential}
In this section, we shall prove Theorem~\ref{thm:exponential} in the introduction.
First, we define ``ideal cones'', determine the corresponding curves in $\LH$ 
and investigate behavior of their mean curvature.
Next, we introduce the notion of developable surfaces {\it of exponential type\/} 
in $\H^3$. Finally, we prove Theorem~\ref{thm:exponential}.

\subsection{Null curves and ideal cones}
\label{subsec:3_idealcones}
\hspace{2mm}

\begin{definition}[Ideal cones]
\label{def:idealcones}
We call a complete developable surface in $\H^3$ an {\it ideal cone\/},
if it is a locus of 1-parameter family of geodesics
sharing one side end as a same point in the ideal boundary.
The shared point is called {\it vertex\/}.
\end{definition}

\begin{proposition}
\label{prop:vertex}
An ideal cone gives a curve in $\LH$ which is null with respect to both
$\mathcal{G}^{\mathfrak{i}}$ and $\mathcal{G}^{\mathfrak{r}}$. 
Conversely, if the locus of a curve in $\LH$ which is null with respect to 
both $\mathcal{G}^{\mathfrak{i}}$ and $\mathcal{G}^{\mathfrak{r}}$
is complete, then the locus is an ideal cone.
\end{proposition}

\begin{proof}
Without loss of generality, we may assume the vertex of the ideal cone 
is $\infty \in \partial \H^3$.
Then the curve 
$\alpha(s)=( \m_1(s), \m_2(s) ) \in (\hat{\C}\times\hat{\C})\setminus\hat{\Delta} = \LH$ 
given by the ideal cone satisfies $\m_2(s)=0$. Hence 
$\mathcal{G}^{\mathfrak{r}}(\alpha',\alpha')
=\mathcal{G}^{\mathfrak{i}}(\alpha',\alpha')=0$ holds.
Conversely, a curve $\alpha(s)=( \m_1(s), \m_2(s) )$ in $\LH$ 
is null with respect to $\mathcal{G}^{\mathfrak{i}}$ 
if and only if 
$\mathcal{G}(\alpha', \alpha')$ is always real.
Moreover if $\alpha$ is null with respect to $\mathcal{G}^{\mathfrak{r}}$, we have
\begin{equation}\label{eq:EQcondition}
  \mathcal{G}(\alpha', \alpha')
  =\frac{\m'_1(s)\bar{\m}'_2(s)}{(1+\m_1(s)\bar{\m}_2(s))^2} 
  = 0,
\end{equation}
for all $s$. By the regularity of $\alpha$, 
\eqref{eq:EQcondition} holds if and only if either $\m'_1(s)$ vanishes identically 
or so does $\m'_2(s)$.
This means the locus of $\alpha$ is a ruled surface which 
is asymptotic to a point in the ideal boundary.
\end{proof}

\begin{remark}
By Proposition \ref{prop:vertex},
it follows that a complete {\it ruled\/} surface 
which is a locus of 1-parameter family of geodesics
sharing one side end as a same point in the ideal boundary 
is necessarily developable, that is, an ideal cone.
If the vertex is $\infty \in \partial \H^3$, the shape of ideal cone is 
a cylinder over a plane curve in the upper half space $\R^3_+$ 
(cf.\ Figure \ref{fig:ballupper}).
\end{remark}

\begin{figure}[htb]
 \begin{center}
 \begin{tabular}{c@{\hspace{30mm}}c}
  \resizebox{3cm}{!}{\includegraphics{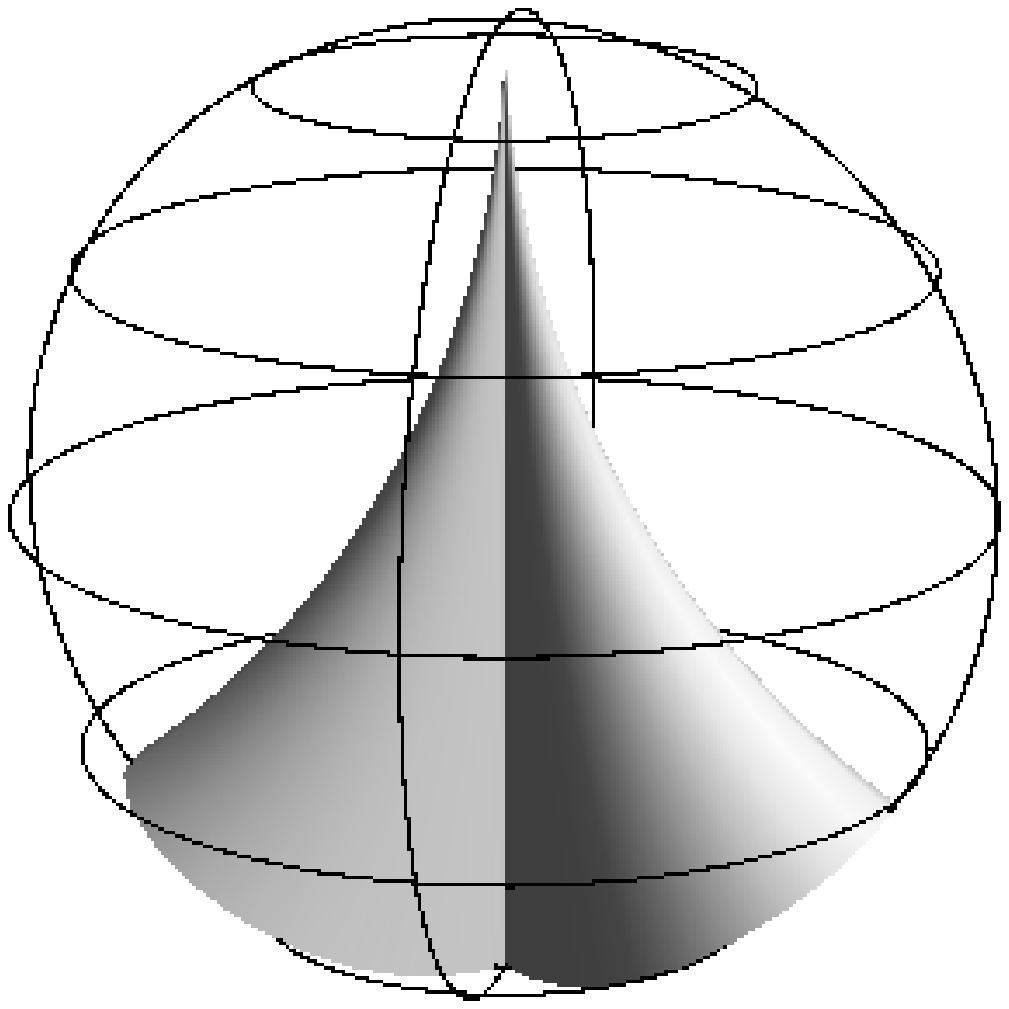}} &
  \resizebox{3cm}{!}{\includegraphics{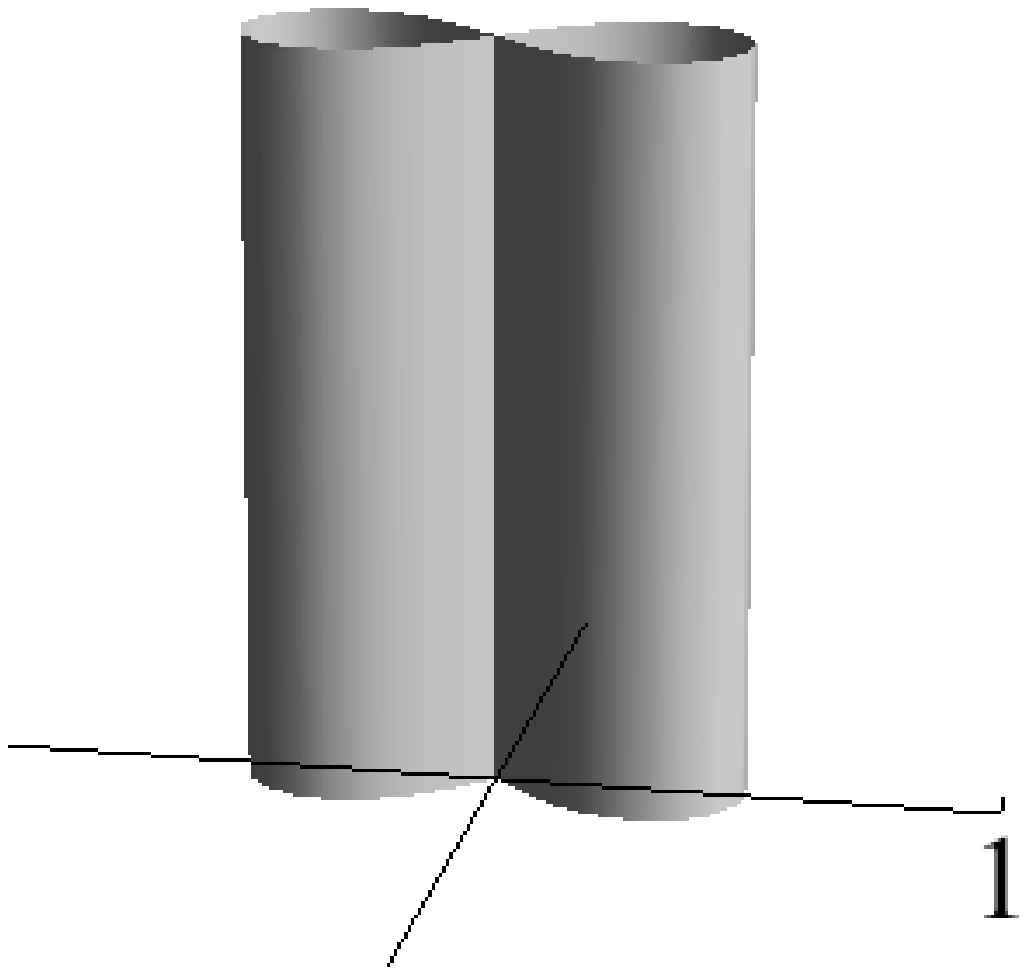}}\\
  {\footnotesize (a) in the Poincar\'e ball model} &
  {\footnotesize (b) in the upper half space model}
 \end{tabular}
 \end{center}
 \caption{An ideal cone whose vertex at $\infty$.}
 \label{fig:ballupper}
\end{figure}

Now we shall investigate behavior of the mean curvature of ideal cones.

\begin{proposition}
\label{prop:idealcone}
For an ideal cone $f$, 
let $\gamma$ be an asymptotic curve of the non umbilic point set of $f$ 
such that $\gamma_+$ is the vertex of $f$,
and let $t$ be the arc length parameter of $\gamma$.
Then the mean curvature $H$ of $f$ is proportional to $e^t$ on $\gamma$.
\end{proposition}

\begin{proof}
Without loss of generality, 
we may assume the vertex of $f$ is $\infty \in \partial \H^3$.
Then the curve $\alpha$ in $\LH$ corresponding to $f$ is given by 
$\alpha(s)=(\m(s),0)$ on $\mathcal{U} \subset \LH$.
By the representation formula \eqref{eq:representation},
$f$ can be written as
\begin{equation}\label{eq:repre-IC}
  f(s,t)=
  \left(
  \begin{array}{cc}
  e^{t}+e^{-t}|\m(s)|^2&
  -e^{-t} \m(s)  \\
  -e^{-t} \bar{\m}(s) &
  e^{-t}
  \end{array}
  \right).
\end{equation}
Then the induced metric $g=f^*\inner{~}{~}$ is
\begin{equation}\label{eq:ind-met} 
  g=e^{-2t} |\m'|^2ds^2+dt^2. 
\end{equation}

Now we shall see that $\m(s)$ can be considered as an 
Euclidean plane curve as follows.
By the isometry $\Psi : \H^3 \rightarrow \R^3_+$ as in \eqref{eq:isometry},
$f$ is transferred to $(\Psi \circ f)(s,t) =( \m(s), e^t ) \in \R^3_+$, that is,
the cylinder over the plane curve $\m(s)\in \C$.
Set $\Omega := \{(w,1) \,|\, w \in \C\} \subset \R^3_+$,
a complete flat surface in $\R^3_+$ so-called the {\it horosphere\/} 
through $(0,1)$ and $\infty$.
Thus $\Omega$ can be considered as the Euclidean plane.
Then the intersection of $f$ and $\Omega$ is parametrized by
$ (\Psi \circ f)(s,0) = ( \m(s), 1 )$.
Thus we can consider $\m$ as a curve in the Euclidean plane $\Omega$.

If we take the arc length parameter $s$ of the curve $\m$ in $\Omega$, 
the induced metric $g$ in \eqref{eq:ind-met} is written as $g=e^{-2t} ds^2+dt^2$.
Since the unit normal vector field $\vect{\nu}$ of $f$ can be expressed by
\[
  \vect{\nu}(s,t)=\left(
  \begin{array}{cc} 
    2\Im(\bar{\m}\m' ) & i\m'\\ -i\bar{\m}' & 0 
  \end{array}\right),
\]
the second fundamental form $I\!I$ of $f$ is written as
$I\!I = e^{-t} \Im(\m' \bar{\m}'')ds^2 = - e^{-t} \kappa_{E}(s) ds^2,$
where $\kappa_{E}$ is the curvature of $\m$ in the Euclidean plane $\Omega$.
Therefore the mean curvature $H$ of $f$ is given by
$H(s,t) = -e^t \kappa_{E}(s) / 2.$
\end{proof}

\subsection{Developable surfaces of exponential type}
\label{subsec:3_exptype}
\hspace{2mm}

Here we shall investigate behavior of the mean curvature 
of {\it complete\/} developable surfaces.
For a complete developable surface $f : \Sigma \rightarrow \H^3$,
let $p \in \Sigma$ be a non umbilic point. 
Then there exists a unique asymptotic curve $\gamma$ 
through $p$ which is a geodesic in $\H^3$.
By hyperbolic Massey's lemma (Lemma~\ref{lem:Massey}), it holds that 
\[
  \frac{1}{H}=P \cosh t + Q \sinh t  
\]
on $\gamma$ (see \eqref{eq:meancurva}), 
where $P$ and $Q$ are constants and $t$ is the arc length parameter of $\gamma$.
Without loss of generality, we may assume $P$ is positive. Then
\begin{equation*}
  \frac{1}{H}= 
  \begin{cases}
    \sqrt{P^2 -Q^2}\cosh\left(t +\dfrac{1}{2}\log \dfrac{P+Q}{P-Q}\right) 
      & \qquad(\text{if}~ P>|Q|), \\
    P e^{\pm t} 
      & \qquad(\text{if}~ P=|Q|),\\
    \sqrt{Q^2 -P^2}\sinh \left(t +\dfrac{1}{2}\log \dfrac{Q+P}{Q-P}\right) 
      & \qquad(\text{if}~ P<|Q|). 
  \end{cases}
\end{equation*}
Completeness of $f$ implies that $t$ varies from $-\infty$ to $\infty$. 
But in the third case,
the mean curvature diverges at some $t \in \R$, a contradiction.
Hence only the first and the second cases can happen,
that is, the mean curvature $H$ of a complete developable surface 
is proportional to exponential function or hyperbolic secant function 
on each asymptotic curves with respect to the arc length parameter.

\begin{definition}[Developable surfaces of exponential type]
\label{def:exptype}
A complete developable surface is said to be {\it of exponential type\/}
if it is not totally umbilic and the mean curvature is proportional to $e^{\pm t}$ 
on each asymptotic curves in the set of non umbilic points, 
where $t$ is the arc length parameter of the asymptotic curve.
\end{definition}

Proposition \ref{prop:idealcone} says that 
non totally umbilic ideal cones are developable surfaces of exponential type.

\subsection{Proof of Theorem \ref{thm:exponential}}
\label{subsec:3_sufficient}
\hspace{2mm}

\begin{definition}[Asymptotics of geodesics]
\label{def:asymp-geod}
Two unit speed geodesics $\gamma_1$, $\gamma_2$ in $\H^3$ are said to be 
{\it asymptotic\/} if 
$ \left\{  d\left( \gamma_1(t),\gamma_2(t)  \right) ~|~ t > 0 \right\} $
is bounded from above, where $d$ denotes the hyperbolic distance.
\end{definition}
For $(p,v)$, $(q,w) \in U\H^3$, it is known that the geodesics
\[
  \gamma_{p,v}(t)=p\cosh t+v\sinh t ,\qquad \gamma_{q,w}(t)=q\cosh t+w\sinh t
\]
are asymptotic if and only if
$ \inner{p+v}{q+w}=0$ holds.

Theorem~\ref{thm:exponential} in the introduction is proved directly by the following

\begin{proposition}\label{prop:exp-IC}
A developable surface of exponential type
whose umbilic point set has no interior is an ideal cone.
That is, asymptotic curves of such a surface are asymptotic to each other.
\end{proposition}

Let $f : \Sigma \rightarrow \H^3$ be a developable surface of exponential type 
whose umbilic point set has no interior. 
We may assume $\Sigma$ is simply connected, 
taking the universal cover $\H^2$, if necessary.
Here, we consider $\H^2$ as the hyperboloid in 
the Lorentz-Minkowski $3$-space $\L^3$.
The proof is divided into three steps (Claims $1$--$3$).

\vspace{2mm}
\noindent
{\bf Claim 1.}~{\it
There exists a global coordinate system 
$\varphi=(s,t) : \Sigma=\H^2 \rightarrow \R^2$ such that
\begin{equation}\label{eq:expimms}
  (f \circ \varphi^{-1})(s,t)= c(s)\cosh t + v(s)\sinh t
\end{equation}
holds, the induced metric $g$ and the second fundamental form $I\!I$ of $f$ 
are given by 
\[
  g = g_{11}(s,t) ds^2 + dt^2,\qquad I\!I= e^{t} \delta(s) g_{11}(s,t) ds^2,
\] 
respectively, where $\delta$ is a smooth function of $s$.
}

\begin{proof}
Since the umbilic point set of $f$ has no interior, 
the proof of Proposition \ref{prop:Portnoy}
implies that each connected component of umbilic point set is a geodesic in $\H^3$.
Thus by the proof of Lemma \ref{lem:Massey}, 
we can find a coordinate neighborhood $(U ; (s,t)) \subset \H^2$ 
such that $U$ is open dense in $\H^2$ and $g = g_{11}(s,t) ds^2 + dt^2$ hold on $U$. 
By taking $t \mapsto t + {\rm constant}$, if necessary, 
each coordinate system $(s,t)$ can be joined smoothly 
over the umbilic point set.
\end{proof}

\vspace{2mm}
\noindent
{\bf Claim 2.}~{\it
The vector field $v(s)$ in \eqref{eq:expimms} is expressed as
\begin{equation}\label{eq:direction}
 v(s) = \frac{ \vect{n}(s)+\delta(s)\vect{b}(s) }{\sqrt{1+\{\delta(s)\}^2}},
\end{equation}
where $\vect{n}$ and $\vect{b}$ denotes 
the principal and binormal normal vector field of the curve ${c}$ in $\H^3$, respectively. 
Furthermore, the curvature $\kappa$ and the torsion $\tau$ of ${c}$ satisfy
\begin{equation}\label{eq:curvature-torsion}
  \kappa(s)=\sqrt{1+\{\delta(s)\}^2}  ,\qquad \tau(s)=\frac{\delta'(s)}{1+\{\delta(s)\}^2}.
\end{equation}
}

\begin{proof}
We may assume the curve $c$ in $\H^3$ is parametrized by the arc length $s$.
Let $\beta$ be the curve in $\H^2$ 
which is the inverse image of the curve ${c}$ by $f$.
By changing the orientation of $\beta$, if necessary,
we may assume the unit normal vector $N$ of $\beta$ in $\H^2$ satisfies
\begin{equation}\label{eq:conormal}
  f_{*}(N)=v.
\end{equation}
Then the map $Y:\R^2 \rightarrow \H^2 \subset \L^3$ defined by
\[
  Y(s,t)=\beta(s)\cosh t + N(s)\sinh t
\]
gives a parametrization of $\H^2$.
Let $\vect{\nu}$ be the unit normal vector field of $f$.
Then the shape operator $A$ of $f$ satisfies 
$A(Y_s) = \delta(s)e^{t} Y_s$, $A(Y_t) = \vect{0}$.
Let $\kappa_{\beta}$ be the geodesic curvature of $\beta$ 
and $\nabla$ the Levi-Civita connection of $\H^2$.
By the Frenet formula for the curve $\beta$ in $\H^2$,
\begin{equation}\label{eq:Frenet}
 \nabla_s{N}=N'(s)=-\kappa_{\beta}(s)\beta'(s) 
\end{equation}
holds, where we consider $N$ is the $\L^3$-valued function and $N'=dN/ds$, etc.
Thus we have 
$Y_s := \partial Y/\partial s = (\cosh t- \kappa_{\beta}(s)\sinh t )\beta'(s)$,
and hence
\[
  \nabla_t Y_s 
  = \frac{\sinh t -\kappa_{\beta}(s)\cosh t}{\cosh t -\kappa_{\beta}(s)\sinh t} Y_s
\]
holds. Since the shape operator $A$ of $f$ satisfies 
the Codazzi equation \eqref{eq:Codazzi}, it follows that
\begin{eqnarray*}
  \vect{0}=(\nabla_tA)(Y_s)-(\nabla_sA)(Y_t)
  =\nabla_t(\delta(s)e^{t} Y_s)
  =\left( 1 + \frac{\sinh t -\kappa_{\beta}(s)\cosh t}{
        \cosh t -\kappa_{\beta}(s)\sinh t} \right) \delta(s)e^{t} Y_s,
\end{eqnarray*}
where $Y_t=\partial Y/\partial t$. Substituting $t=0$ into this, we have that 
\begin{equation}\label{eq:horocycle}
  \kappa_{\beta}(s) = 1
\end{equation}
for $s$ in $\R$, that is, $\beta$ is congruent to the horocycle. 

Next, we shall calculate the principal normal vector field $\vect{n}$, 
the binormal vector field $\vect{b}$,
curvature $\kappa$ and torsion $\tau$ of the curve $c$ in $\H^3$.
Let $D$ be the Levi-Civita connection of $\H^3$.
By \eqref{eq:Frenet} and \eqref{eq:horocycle}, 
$\nabla_s \beta'(s)=N(s)$ holds. Moreover, by \eqref{eq:conormal}, it holds that
\begin{eqnarray*}   
D_s{c}'(s)
&=&f_{\ast}(\nabla_s \beta'(s))+I\!I(\beta'(s),\beta'(s)) \vect{\nu}(s,0)\\
&=&f_{\ast}(N(s))+\delta(s)\vect{\nu}(s,0)
={v}(s)+\delta(s)\vect{\nu}(s,0),
\end{eqnarray*}
and hence we have
\[
  \kappa(s)
  =\left| D_sc'(s) \right|
  =\sqrt{1+\{\delta(s)\}^2},
\qquad
  \vect{n}(s)
  =\frac{D_sc'(s)}{\kappa(s)}
  =\frac{{v}(s)+\delta(s)\vect{\nu}(s,0)}{\sqrt{1+\{\delta(s)\}^2}}.
\]
If we denote by $\vect{e}(s)=c' (s)$ the unit tangent vector field of ${c}$,
$\vect{b}(s)$ is obtained as
\begin{equation*}
  \vect{b}(s)
  =\vect{e}(s) \times \vect{n}(s)
  =\frac{\vect{\nu}(s,0)-\delta(s){v}(s)}{\sqrt{1+\{\delta(s)\}^2}},
\end{equation*}
where $\times$ is the cross product in $\H^3$ (cf.\ \eqref{eq:cross}). Since
\[
  \left\{
  \begin{array}{ll}
  D_s\vect{\nu}(s,0)
  =-f_{\ast}(A(Y_s)(s,0))
  =-f_{\ast}(\delta(s)Y_s(s,0))
  =-\delta(s)\vect{e}(s)\\
  D_s{v}(s)
  =-f_{\ast}(\nabla_sN)- \inner{A(N)}{\beta'} \vect{\nu}(s,0)
  =f_{\ast}(-\beta'(s))=-\vect{e}(s),
  \end{array}
  \right.
\]
we have
\begin{eqnarray*}
  D_s\vect{b}(s)=\vect{b}'(s)
  =-\frac{\delta'(s)}{1+\{\delta(s)\}^2}\frac{{v}(s)+\delta(s)\vect{\nu}(s,0)}{%
       \sqrt{1+\{\delta(s)\}^2}}
  =-\frac{\delta'(s)}{1+\{\delta(s)\}^2}\vect{n}(s).
\end{eqnarray*}
Thus the torsion $\tau$ of ${c}$ is given as in \eqref{eq:curvature-torsion}.
Since the unit vector field $v(s)$ is included in the normal plane of ${c}$ and satisfies 
\[
  \langle {v}(s),\vect{n}(s) \rangle =\frac{1}{\sqrt{1+\{\delta(s)\}^2}} ,\qquad
  \langle {v}(s),\vect{b}(s) \rangle = -\frac{\delta(s)}{\sqrt{1+\{\delta(s)\}^2}}, 
\]
we have that $v(s)$ is the form given in \eqref{eq:direction}.
\end{proof}

\vspace{2mm}
\noindent
{\bf Claim 3.}~{\it
Any two asymptotic curves are asymptotic to each other in the sense of 
Definition $\ref{def:asymp-geod}$.
}

\begin{proof}
Under the notations in Claim 1 and 2, we have
\[
  (f \circ \varphi^{-1})(s,t)
  = {c}(s)\cosh t + \frac{ \vect{n}(s)+\delta(s)\vect{b}(s) }{\kappa(s)} \sinh t .
\]
For $s \in \R$, set $\gamma_{s}(t):=(f \circ X)(s,t)$.
It is sufficient to prove that, for fixed $s_0 \in \R$, the function
\[ 
  \rho : \R \ni s \longmapsto 
  \inner{{c}(s)+\frac{ \vect{n}(s)+\delta(s)\vect{b}(s) }{\kappa(s)}}{{c}(s_0)
  +\frac{ \vect{n}(s_0)+\delta(s_0)\vect{b}(s_0) }{\kappa(s_0)}} \in \R,
\]
is equivalently zero. Using the Frenet-Serret formula 
\[
  \vect{e}'(s)=c(s)+\kappa(s)\vect{n}(s),\qquad 
  \vect{n}'(s)=-\kappa(s)\vect{e}(s)+\tau(s)\vect{b}(s),\qquad 
  \vect{b}'(s)=-\tau(s)\vect{n}(s) 
\]
for the curve ${c}$ in $\H^3$, we have
\begin{multline}\label{eq:phi}
  \frac{d}{ds}\left(c(s)+\frac{ \vect{n}(s)+\delta(s)\vect{b}(s) }{\kappa(s)}\right) =
  \frac{\kappa(s)\tau(s)\delta(s)-\kappa'(s)}{\kappa^2(s)}\vect{n}(s)\\
  +\frac{\kappa(s)\tau(s)-\kappa(s)\delta'(s)+\kappa'(s)\delta(s)}{\kappa^2(s)}\vect{b}(s).
\end{multline}
On the other hand, we have
\[
  \kappa(s)\tau(s)\delta(s)-\kappa'(s)
  =\kappa(s)\tau(s)-\kappa(s)\delta'(s)+\kappa'(s)\delta(s)
  =0,
\]
by \eqref{eq:curvature-torsion} in Claim $2$.
Substituting this into \eqref{eq:phi}, 
we have $\rho'(s) = 0$ for all $s$. 
Besides $\rho(s_0)=0$, we obtain $\rho(s) = 0$ for all $s$.
\end{proof}

\subsection{A non-real-analytic example}
\label{subsec:ex-exp}
\hspace{2mm}

\begin{example}
\label{ex:NRA}
The assumption of analyticity in Theorem \ref{thm:exponential} cannot be 
removed since non-real-analytic developable surfaces of exponential type 
might have more than one asymptotic points.
Figure~\ref{fig:expn2} shows an example asymptotic to distinct two points 
in the ideal boundary.

\begin{figure}[htb]
 \begin{tabular}{cc}
  \resizebox{4cm}{!}{\includegraphics{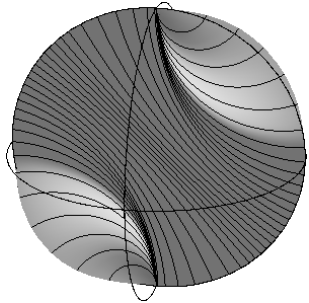}} 
 \end{tabular}  
 \caption{
      A non-real-analytic developable surface of 
      exponential type asymptotic to $0$ and $\infty$.}
 \label{fig:expn2}
\end{figure}
\end{example}

\noindent
The corresponding curve $\alpha(s)$ in $\LH$ is given by 
$\alpha(s)=(x_1(s)+iy_1(s),x_2(s)+iy_2(s))$, where

\[
  x_1(s)= 
  \begin{cases}
    0
    & (s \leq -1) \\
    (\sqrt{2}-1)(s+1)/(1+e^{\frac{1}{s}+\frac{1}{s+1}}) 
    &(-1<s<0)\\
    (\sqrt{2}-1)(s+1)
    & (0\leq s), 
  \end{cases} \qquad 
  y_1(s)= 
  \begin{cases}
    0
    & (s \leq \sqrt{2}) \\
    2 e^{\frac{\sqrt{2}+1}{\sqrt{2}-s}} 
    & (\sqrt{2}<s),
  \end{cases}
\]
\[
  x_2(s)= 
  \begin{cases}
    (\sqrt{2}-1)(1-s) 
    & (s \leq 0) \\
    (\sqrt{2}-1)(1-s)/(1+e^{\frac{1}{1-s}-\frac{1}{s}}) 
    &(0<s<1)\\
    0
    & (1\leq s), 
  \end{cases} \qquad 
  y_2(s)= 
  \begin{cases}
    2 e^{\frac{\sqrt{2}+1}{\sqrt{2}-s}}
    & (s \leq -\sqrt{2}) \\
    0
    & (-\sqrt{2}<s).
  \end{cases}
\]


\end{document}